\documentclass[twoside]{article}

\usepackage{bbm}

\usepackage{amsmath}
\usepackage{times}
\usepackage{amssymb}
\usepackage{amsthm}
\usepackage{bm}
\usepackage{graphicx}
\usepackage{caption}
\usepackage{geometry}
\usepackage[round]{natbib}
\usepackage{color}
\usepackage{subfigure}
\usepackage{comment}
\usepackage{enumitem}
\usepackage{mathtools}
\usepackage{thmtools, thm-restate}
\usepackage{verbatim}

\usepackage[colorlinks,urlcolor=blue,citecolor=blue,linkcolor=blue]{hyperref}
\newcommand{\footremember}[2]{%
    \footnote{#2}
    \newcounter{#1}
    \setcounter{#1}{\value{footnote}}%
}
\newcommand{\footrecall}[1]{%
    \footnotemark[\value{#1}]%
}

\usepackage{xcolor}
\definecolor{myblue}{rgb}{0,0.08,0.45}

\usepackage{hyperref}
\usepackage[utf8]{inputenc} 
\usepackage[T1]{fontenc}    
\usepackage{url}            
\usepackage{booktabs}       
\usepackage{nicefrac}       
\usepackage{microtype}      

\DeclareMathOperator*{\argmin}{arg\,min}

\DeclareMathOperator*{\dist}{dist}

\DeclareMathOperator*{\diag}{diag}

\usepackage{thmtools}
\usepackage{thm-restate}

\definecolor{mypurple}{rgb}{0.67, 0.12, 0.47}
\newcommand{\jd}[1]{{\color{mypurple}{\textbf{JD:} #1}}}

\newcommand{\xc}[1]{{\color{brown}{\textbf{XC:} #1}}}

\usepackage{booktabs}       
\usepackage{nicefrac}       
\usepackage{microtype}      

\usepackage{dsfont}

\usepackage{algorithm,algorithmic}
\usepackage[algo2e,ruled,vlined]{algorithm2e}
\usepackage{graphicx}
\usepackage{subfigure}

\newtheorem{theorem}{Theorem}

\newtheorem{remark}{Remark}

\newtheorem{assumption}{Assumption}

\def\1{\bm{1}}

\def\rr{{\mathbb{R}}}

\def\vzero{{\bm{0}}}

\def\vg{{\bm{g}}}

\def\vx{{\bm{x}}}
\def\vy{{\bm{y}}}

\def\mA{{\bm{A}}}

\def\mI{{\bm{I}}}

\def\mQ{{\bm{Q}}}

\def\mLambda{{\bm{\Lambda}}}

\def\mQh{\widehat{\bm{Q}}}
\def\mQt{\widetilde{\bm{Q}}}

\DeclareMathAlphabet{\mathsfit}{\encodingdefault}{\sfdefault}{m}{sl}
\SetMathAlphabet{\mathsfit}{bold}{\encodingdefault}{\sfdefault}{bx}{n}

\def\gB{{\mathcal{B}}}

\def\gD{{\mathcal{D}}}

\def\gF{{\mathcal{F}}}

\def\gS{{\mathcal{S}}}

\def\sR{{\mathbb{R}}}

\newcommand{\E}{\mathbb{E}}

\newcommand{\innp}[1]{\left\langle #1 \right\rangle}
\usepackage[colorinlistoftodos]{todonotes}

\usepackage{psfrag}

\usepackage{multicol}
\usepackage{multirow}

\usepackage{balance}
\usepackage{threeparttable}

\newcommand{\ee}{\mathbb{E}}
\title{Cyclic Block Coordinate Descent With Variance Reduction \\
for Composite Nonconvex Optimization}
\author{Xufeng Cai\footremember{wisc}{Department of Computer Sciences, University of Wisconsin-Madison. XC (\href{mailto:xcai74@wisc.edu}{xcai74@wisc.edu}), CS (\href{mailto:chaobing.song@wisc.edu}{chaobing.song@wisc.edu}), SJW (\href{mailto:swright@cs.wisc.edu}{swright@cs.wisc.edu}), JD (\href{mailto:jelena@cs.wisc.edu}{jelena@cs.wisc.edu}).}
\and Chaobing Song\footrecall{wisc}
\and Stephen J. Wright\footrecall{wisc}
\and Jelena Diakonikolas\footrecall{wisc}
}
\date{}

\begin{document}
\maketitle
%

%

\begin{abstract}
    Nonconvex optimization is central in solving many machine learning problems, in which block-wise structure is commonly encountered. In this work, we propose cyclic block coordinate methods for nonconvex optimization problems with non-asymptotic gradient norm guarantees. Our convergence analysis is based on a gradient Lipschitz condition with respect to a Mahalanobis norm, inspired by a recent progress on cyclic block coordinate methods. In deterministic settings, our convergence guarantee matches the guarantee of (full-gradient) gradient descent, but with the gradient Lipschitz constant being defined w.r.t.~a Mahalanobis norm. In stochastic settings, we use recursive variance reduction to decrease the per-iteration cost and match the arithmetic operation complexity of current optimal stochastic full-gradient methods, with a unified analysis for both finite-sum and infinite-sum cases. We prove a faster linear convergence result when a Polyak-{\L}ojasiewicz (P{\L}) condition holds. To our knowledge, this work is the first to provide non-asymptotic convergence guarantees --- variance-reduced or not --- for a cyclic block coordinate method in general composite (smooth + nonsmooth) nonconvex settings. Our experimental results demonstrate the efficacy of the proposed cyclic scheme in training deep neural nets. 
\end{abstract}

\section{Introduction}

Exploiting structural information in machine learning (ML) problems is key to enabling optimization at extreme scale. 
Important examples of such structure are block separability, giving rise to block coordinate methods, and finite/infinite sum structure, giving rise to stochastic, possibly variance-reduced optimization methods. 
In this work, we explore both these types of structure to develop novel optimization methods with fast convergence. 

We focus on nonconvex optimization problems of the form
%
%
\begin{align}
\min_{\vx\in\sR^d} F(\vx)  = f(\vx) + r(\vx),     \label{eq:prob}
\end{align}
where 
$\vx \in\sR^d$ can be partitioned into $m$ disjoint blocks $\vx = (\vx^1, \dotsc, \vx^m)$ with $\vx^j\in \sR^{d_j}$ for $j \in [m]$ and $\sum_{j=1}^m d_j = d$; $f(\vx)$ is a smooth nonconvex function; $r(\vx) = \sum_{j=1}^m r^j(\vx^j)$ is block separable, extended-valued, closed convex function such that each $r^j(\cdot)$ (and thus the separable sum $r$) admits an efficiently computable proximal operator. 
We consider in particular  the finite-sum variant of \eqref{eq:prob}:
\begin{align}
\min_{\vx \in \rr^d} F(\vx) = f(\vx) +  r(\vx) = \frac{1}{n}\sum_{i=1}^n f_i(\vx) + r(\vx), \label{eq:sto-opt}
\end{align}
in which $f(\vx)$ is nonconvex and smooth and $n$ is usually very large. Without loss of generality, due to the central limit theorem, we use $n = +\infty$ in \eqref{eq:sto-opt} to refer to the following stochastic (infinite-sum) optimization setting: 
\begin{align}\label{pro:stoc}
\min_{\vx \in \rr^d} F(\vx) = \E_{\xi\sim\gD}[f(\vx; \xi)] + r(\vx),
\end{align}
where $\xi$ is a random variable from an unknown distribution $\gD$. 
Problems of the form~\eqref{eq:sto-opt} and \eqref{pro:stoc} commonly arise in machine learning, especially in (regularized, empirical, or population) risk minimization.  
 
\subsection{Motivation and Related Works}

Both block coordinate and variance-reduced stochastic gradient methods are prevalent in machine learning, due to their effectiveness in handling large problem instances; see e.g.,~\citet{gorbunov2020unified,wright2015coordinate,allen2016even,nesterov2012efficiency,allen2017katyusha,johnson2013accelerating,diakonikolas2018alternating,nakamura2021block,li2020page,beck2013convergence,hong2017iteration,xu2015block,chen2016accelerated} and references therein. 

Block coordinate  methods can be classified into three main categories according to the order in which blocks of coordinates are selected: 
(i) greedy, or Gauss-Southwell methods~\citep{nutini2015coordinate}, which in each iteration selects the block of coordinates that lead to the highest progress in minimizing the objective function; 
(ii) randomized block coordinate methods, which select blocks of coordinates at random (with replacement), according to some pre-defined probability distribution~\citep{nesterov2012efficiency}; and 
(iii) cyclic block coordinate methods, which update the coordinate blocks in a cyclic order~\citep{beck2013convergence}.
(A combination of (ii) and (iii) known as random-permutations methods uses a cyclic approach but randomly reshuffles the order in which the blocks are updated at the start of each cycle.)
Greedy methods can be quite effective in practice when their selection rule can be implemented efficiently, but they are applicable only to very specialized problems. 
Thus, most of the focus has been on randomized and cyclic methods. 

From a theoretical standpoint, randomized methods have received much more attention than cyclic methods. The reason is that the randomly selected block of gradient coordinates can be related to the full gradient by taking the expectation, which allows their analysis to be reduced to the related to the analysis of standard first-order methods; see 
\citet{nesterov2012efficiency,nesterov2017efficiency,allen2016even,diakonikolas2018alternating}. 
By contrast, cyclic methods are much more challenging to analyze, as it is unclear how to relate the partial gradient to the full one. 
Obtaining non-asymptotic convergence guarantees for such methods was initially considered nearly impossible~\citep{nesterov2012efficiency}. 
Despite much of the progress on the theoretical front~\citep{beck2013convergence,saha2013nonasymptotic,gurbuzbalaban2017cyclic,lee2019random,wright2020analyzing,li2017faster,sun2021worst}, most of the literature addressing cyclic methods deals with convex (often quadratic) objective functions and provides convergence guarantees that are typically worse by a factor polynomial in the dimension $d$ than the equivalent guarantees for randomized methods. 
For nonconvex objectives, there are few existing guarantees, and these require additional assumptions such as multiconvexity (i.e., that the function is convex over a coordinate block when other blocks of coordinates remain fixed) and the Kurdyka-{\L}ojasiewicz (K{\L}) property, or else provide convergence guarantees that are only asymptotic~\citep{xu2013block,xu2015block,xu2017globally, zeng2014cyclic}. 
On the other hand, the recent work by \citet{song2021fast} avoids the explicit dependence on the dimension by introducing a novel Lipschitz condition that holds w.r.t.~a Mahalanobis norm. 
This condition is the inspiration for the methods proposed in our work. 
The techniques in \citet{song2021fast} cannot be applied directly to the current context of composite nonconvex problems, as they address monotone variational inequalities. An entirely separate analysis framework is required, and is presented here.













From the implementation viewpoint, randomized methods require generating pseudo-random numbers from a pre-defined probability distribution to determine which coordinate block should be selected in each iteration.
This operation may dominate the arithmetic cost when the coordinate update is cheap. 
Cyclic methods are simple, intuitive, and more efficient for implementation, and often demonstrate better empirical performance than the randomized methods~\citep{beck2013convergence, chow2017cyclic, sun2021worst}.
They are thus the default algorithms for many software packages such as SparseNet~\citep{mazumder2011sparsenet} and GLMNet~\citep{friedman2010regularization} in high-dimensional computational statistics and have found wide applications in areas such as variational inference~\citep{blei2017variational, plummer2020dynamics}, non-negative matrix factorization~\citep{vandaele2016efficient}, $k$-means clustering~\citep{nie2021coordinate}, and phase retrieval~\citep{zeng2020coordinate}. 
More recent literature has also sought to combine the favorable properties of stochastic optimization methods (such as SGD) with block coordinate updates; see, e.g.,~\citet{xu2015block, nakamura2021block, fu2020block, chen2016accelerated, lei2020asynchronous, wang2016randomized}), which address nonconvex problems of the form \eqref{eq:sto-opt} and \eqref{pro:stoc}. Compared with traditional stochastic gradient methods, which simultaneously update all variables using Gauss-Jacobi-style iterations, block-coordinate variants of stochastic gradient update the variables sequentially with Gauss-Seidel-style iterations, thus usually taking fewer iterations to converge (see e.g.,~\citet{xu2015block}). One common approach to further improve sample complexity in stochastic optimization is to use variance reduction, which for block coordinate methods in nonconvex settings has been done in~\citet{chen2016accelerated, chauhan2017mini, lei2020asynchronous}. However, to the best of our knowledge, non-asymptotic convergence results have only been established for randomized methods with variance reduction~\citep{chen2016accelerated, lei2020asynchronous}. 
We are not aware of work that incorporates variance reduction techniques with cyclic methods in nonconvex settings. Even in the {\em convex} setting, the only work we are aware of that combines variance reduction with a cyclic method is~\citet{song2021fast}, but this paper utilizes SVRG-style variance reduction, whose applicability in nonconvex settings is unclear.

\subsection{Contributions}
Our main contributions can be summarized as follows.

\paragraph{Proximal Cyclic block Coordinate Descent (P-CCD).} We provide a non-asymptotic convergence analysis for the standard P-CCD method in deterministic nonconvex settings, based on a Lipschitz condition w.r.t.~a Mahalanobis norm, inspired by the recent work by~\citet{song2021fast}. However, the techniques are completely disjoint and their results (for monotone variational inequalities) neither imply ours (for nonconvex minimization), nor the other way around. Our Lipschitz condition, which implies block (coordinate) smoothness, is more general. The comparison between the new Lipschitz condition and the standard (Euclidean-norm) Lipschitz condition is discussed in Section~\ref{sec:prelim}. 
We show that P-CCD has the same sublinear convergence rate as full gradient methods, and achieves linear convergence under a P{\L} condition.
To the best of our knowledge, these are the first such results for a cyclic method in the composite nonconvex setting~\eqref{eq:prob}, where standard tools such as monotonicity (convex inequalities) used in the earlier paper cannot be used to establish convergence.

\paragraph{Variance-Reduced P-CCD.} We propose a stochastic gradient variant of P-CCD with recursive variance reduction for solving nonconvex problems of the form~\eqref{eq:sto-opt} and \eqref{pro:stoc}. The recursive variance reduction technique of \citet{li2020page} was used prior to our work only in the full-gradient setting, and the extension to the cyclic block coordinate setting requires addressing nontrivial technical obstacles such as establishing a new potential function and controlling additional error terms arisen from the cyclic update rule.
We prove its non-asymptotic convergence using an analysis that unifies the finite-sum and infinite-sum settings, which also matches the arithmetic operation complexity of optimal stochastic full-gradient methods for nonconvex minimization. 
A faster, linear convergence rate is attained under a P{\L} condition. To our knowledge, our work is the first to incorporate variance reduction into cyclic methods in nonconvex settings while providing non-asymptotic convergence guarantees. 

\paragraph{Numerical Experiments.} We apply our proposed cyclic algorithms to train LeNet on the CIFAR-10 dataset, and compare them with SGD and the PAGE algorithm~\citep{li2020page}. Our preliminary results demonstrate that the cyclic methods converge faster with better generalization than full gradient methods when using large batch sizes, thus shedding light on the possibility of remedying the drawbacks of large-batch methods~\citep{keskar2017on}. 

\subsection{Further Related Work}
Both block coordinate methods and variance reduction techniques in stochastic optimization have been subjects of much research.
For conciseness, we review only the additional literature that is most closely related to our work. 

\paragraph{Block Coordinate Descent.} Block coordinate methods have been widely used in both convex and nonconvex applications such as feature selection in high-dimensional computational statistics~\citep{wu2008coordinate, friedman2010regularization, mazumder2011sparsenet} and empirical risk minimization in machine learning~\citep{nesterov2012efficiency, lin2015accelerated, allen2016even, alacaoglu2017smooth, diakonikolas2018alternating, xu2015block}. 
The convergence of block coordinate methods has been extensively studied for various settings, see e.g.,~\cite{grippof1999globally, tseng2001convergence, razaviyayn2013unified, xu2015block, song2021fast} and  references therein. 
In nonconvex settings,  asymptotic convergence of block coordinate methods was established in~\cite{chen2021global,xu2017globally}. 
In terms of non-asymptotic convergence guarantees, \cite{chen2016accelerated} provides such a result for a randomized method under a sparsity constraint and restricted strong convexity, while~\cite{xu2017globally,xu2013block} provides results for cyclic methods under the K{\L} property. 
For a stochastic gradient variant of a cyclic method, \emph{asymptotic} convergence was analyzed by~\cite{xu2015block}. 

\paragraph{Variance Reduction.} To address the issue of the constant variance of the (minibatch) gradient estimator, several variance reduction methods have been proposed.
SAG~\citep{schmidt2017minimizing} was the first stochastic gradient method with a linear convergence rate for strongly convex finite-sum problems, and was based on a biased gradient estimator. \citet{johnson2013accelerating} and \citet{defazio2014saga} improved SAG by proposing unbiased estimators of SVRG-type and SAGA-type, respectively. These estimators were further enhanced with Nesterov acceleration~\citep{allen2017katyusha, song2020variance} and applied to nonconvex finite-sum/infinite-sum problems \citep{reddi2016stochastic,lei2017non}. 
For nonconvex stochastic (infinite-sum) problems, the recursive variance reduction estimators SARAH \citep{nguyen2017sarah} and SPIDER \citep{fang2018spider,zhou2018finding,zhou2018stochastic} were proposed to attain the optimal oracle complexity of $\mathcal{O}(1/\epsilon^3)$ for finding an {$\epsilon$-approximate} stationary point. 
PAGE~\citep{li2020page} and STORM~\citep{cutkosky2019momentum} further simplified SARAH and SPIDER by reducing the number of loops and avoiding large minibatches. 


\section{Preliminaries}\label{sec:prelim}
We consider a real $d$-dimensional Euclidean space $(\rr^d, \|\cdot\|)$, where $\|\cdot\| = \sqrt{\innp{\cdot, \cdot}}$ is induced by the (standard) inner product associated with the space and $d$ is finite. For any given positive integer $m$, we use $[m]$ to denote the set $\{1, 2, \dots, m\}$. We assume that we are given a positive integer $m \leq d$ and a partition of the coordinates $[d]$ into nonempty sets $\gS^1, \gS^2, \dots, \gS^m$. We let $\vx^j$ denote the subvector of $\vx$ indexed by the coordinates contained in $\gS^j$ and let $d^j := |\gS^j|$ denote the size of the set $\gS^j,$ for $j \in [m].$ To simplify the notation, we assume that the partition into sets $\gS^1, \gS^2, \dots, \gS^m$ is ordered, in the sense that for $1 \leq j< j' \leq m,$ $\max_{i \in \gS^j} i < \min_{i' \in \gS^{j'}} i'.$ This assumption is without loss of generality, as our results are invariant to permutations of the coordinates.  
Given a matrix $\mA$, we let $\|\mA\| := \sup\{\mA\vx: \vx \in \rr^d, \|\vx\| \leq 1\}$ denote the standard operator norm. For a positive definite symmetric matrix $\mA,$ $\|\cdot\|_{\mA}$ denotes the Mahalanobis norm defined by $\|\vx\|_{\mA} = \sqrt{\langle \mA\vx, \vx \rangle}.$ We use $\mI_{d}$ to denote the identity matrix of size $d\times d;$ when the context is clear, we omit the subscript. 
For a sequence of positive semidefinite $d \times d$ matrices $\{\mQ^j\}_{j = 1}^m$, we define $\mQh^j$ by
$$
    (\mQh^j)_{t, k} = 
    \begin{cases}
            (\mQ^j)_{t, k}, & \text{ if } \min\{t, k\} > \sum_{\ell=1}^{j-1} d^{\ell},\\
            0, & \text{ otherwise,}
    \end{cases}
$$
which corresponds to the matrix $\mQ^j$ with first $j - 1$ blocks of rows and columns set to zero. Similarly, we define $\mQt^j$ by
$$
    (\mQt^j)_{t, k} = 
    \begin{cases}
            (\mQ^j)_{t, k}, & \text{ if } \max\{t, k\} \le \sum_{\ell=1}^{j-1} d^{\ell},\\
            0, & \text{ otherwise. }
    \end{cases}
$$
In other words, $\mQt^j$ corresponds to $\mQ^j$ with all but its first $j - 1$ blocks of rows and columns set to zero.

We use $\nabla^j f(\vx)$ to denote the subvector of the gradient $\nabla f(\vx)$ indexed by the elements of $\gS^j$. For a block-separable convex function $r(\vx) = \sum_{j = 1}^m r^j(\vx^j)$, we use 
$r'(\vx)$ and $r^{j, '} (\vx^j)$ to denote the elements in the subdifferential sets $\partial r(\vx)$ and $\partial r^{j}(\vx^j)$ for $j \in [m]$, respectively.

Throughout the paper, we make use of the following assumptions. The first assumption is standard and rules out degenerate problem instances. 

\begin{assumption}
$F(\vx)$ is bounded below and $\vx^*$ is a global minimum of $F.$
\end{assumption}


The following two assumptions are the gradient Lipschitz conditions used in the analysis of our algorithms. These conditions are not standard, due to the choice of weighted norms $\|\cdot\|_{\mLambda_j}, \|\cdot\|_{\mLambda_j^{-1}},$ and $\|\cdot\|_{\mQ^j}.$ Assumption~\ref{assmpt:Q-hat} is inspired by a similar Lipschitz condition introduced by~\citet{song2021fast}, the main difference with that paper being to use a more general norm $\|\cdot\|_{\mLambda_j^{-1}}$ for the gradients. 

\begin{assumption}\label{assmpt:coordinate-smooth}
For all $\vx$ and $\vy$ that differ only in the $j$\textsuperscript{th} block, where $j\in[m]$, $f(\cdot)$ satisfies the following: 
\begin{align}
\|\nabla^j f(\vx) -  \nabla^j f(\vy)\|_{\mLambda_j^{-1}}\le \|\vx^j-\vy^j\|_{\mLambda_j}, \label{eq:coordinate-smooth}
\end{align}
where $\mLambda_j\in\sR^{d_j\times d_j}$ is a positive definite diagonal matrix.
\end{assumption}
Observe that when $\mLambda_j = {L_j}\mI_{d^j}$, Assumption~\ref{assmpt:coordinate-smooth} becomes the standard block Lipschitz condition~\citep{nesterov2012efficiency}. 

\begin{assumption}\label{assmpt:Q-hat}
There exist symmetric positive semidefinite $d \times d$ matrices $\mQ^j,$ $1\leq j \leq m,$ such that each $\nabla^j f(\cdot)$ is $1$-Lipschitz continuous w.r.t.~the seminorm $\|\cdot\|_{\mQ^j}$. That is,  $\forall \vx, \vy \in \sR^d$, we have
\begin{equation}\label{eq:block-Lipschitz}
\|\nabla^j f(\vx) - \nabla^j f(\vy)\|_{\mLambda_j^{-1}}^2 \leq 
\|\vx - \vy\|_{\mQ^j}^2.
\end{equation}
%

\end{assumption}
We remark that matrices $\mQ^j$ do \emph{not} need to be known to the algorithm. 
Observe that if $f$ is $L$-smooth w.r.t.~the Euclidean norm, i.e., if $\|\nabla f(\vx) - \nabla f(\vy)\| \leq L\|\vx - \vy\|$, then Assumption~\ref{assmpt:Q-hat} can be satisfied with $\mLambda_j = L\mI_{d^j}$ and $\mQ^j = L \mI_{d}$ for $j \in [m]$. Indeed,  in this case we have,  
 $\forall \vx, \vy \in \rr^d$, 
$
    \|\nabla^j f(\vx) - \nabla^j f(\vy)\|_{\mLambda_j^{-1}}^2 
    \leq \frac{1}{L}\|\nabla f(\vx) - \nabla f(\vy)\|^2 \leq L\|\vx - \vy\|^2 = \|\vx - \vy\|_{\mQ^j}^2.
$
However, the more general matrices $\mLambda_j$ and $\mQ^j$ in Assumptions~\ref{assmpt:coordinate-smooth} and \ref{assmpt:Q-hat} provide more flexibility in exploiting the problem geometry. For further discussion and comparison to Euclidean Lipschitz constants, see~\citet{song2021fast}. 

In the following, we let $\mLambda$ be the diagonal matrix composed of positive diagonal blocks $\mLambda_j$ for $j \in [m]$, i.e., $\mLambda = \diag(\mLambda_1, \mLambda_2, \ldots, \mLambda_m)$ and provide the appropriate assumption about the P{\L}-condition w.r.t.~the norm $\|\cdot\|_{\mLambda}$. This assumption is used only when proving linear convergence of our algorithms, not throughout the paper. 
The constant in this assumption need not be known. 
\begin{assumption}\label{assmpt:PL-condition}
We say that $F$ satisfies the P{\L} condition w.r.t. $\|\cdot\|_{\mLambda}$ with parameter $\mu >0$, 
if for all $\vx \in \sR^d$, 
\begin{equation}
    {\dist}^2(\partial F(\vx), \vzero) \geq 2\mu\big(F(\vx) - F(\vx^*)\big),
\end{equation}
with $\dist^2(\partial F(\vx), \vzero) := \underset{r'(\vx) \in \partial r(\vx)}{\inf}\|\nabla f(\vx) + r'(\vx)\|_{\mLambda^{-1}}^2$. 
\end{assumption}

\paragraph{Stochastic settings.}
\indent In the stochastic setting of our problem, we consider the finite sum nonconvex optimization problem described by~\eqref{eq:sto-opt}. 
Our analysis also handles the case of stochastic optimization problems of the form \eqref{pro:stoc} by taking $n \rightarrow \infty$. 
To avoid using separate notation for the two settings (finite and infinite sum), we state the assumptions and the results for problems \eqref{eq:sto-opt} and treat \eqref{pro:stoc} as the limiting case of \eqref{eq:sto-opt} when $n \rightarrow \infty$.  
\begin{assumption}\label{assmpt:finite-var}
For any $\vx \in \sR^d$,  
\begin{align}
\E_i\big[ \|\nabla f_i(\vx) -  \nabla f(\vx)\big\|^2_{\mLambda^{-1}} \big] \le \sigma^2, 
\end{align}
where $i$ is drawn uniformly at random from $[n].$ 
\end{assumption}

\begin{assumption}\label{assmpt:coordinate-smooth-sto}
For all $\vx$ and $\vy$ that only differ in the $j$\textsuperscript{th} block for $j\in[m]$,  
\begin{align}
\E_i\big[\|\nabla^j f_i(\vx) -  \nabla^j f_i(\vy)\|_{\mLambda_j^{-1}}\big]\le \|\vx^j-\vy^j\|_{\mLambda_j},
\end{align}
where $i$ is drawn uniformly at random from $[n]$ and  $\mLambda_j \in \sR^{d_j\times d_j}$ is a positive definite diagonal matrix.
\end{assumption}
Assumption~\ref{assmpt:coordinate-smooth-sto} implies Assumption~\ref{assmpt:coordinate-smooth}, due to the finite-sum assumption and uniform sampling. For simplicity, we use the same matrix  $\mLambda_j$ for both smoothness conditions.


\begin{assumption}\label{assmpt:vr-Q-hat-tilde}
There exist positive semidefinite matrices $\mQ^j,$ $1\leq j \leq m$ such that each $\nabla^j f$ is expected $1$-Lipschitz continuous w.r.t.~the seminorm $\|\cdot\|_{\mQ^j},$ i.e., $\forall \vx, \vy \in \sR^d,$
\begin{equation}\label{eq:block-Lipschitz-vr}
\E_i\big[\|\nabla^j f_i(\vx) - \nabla^j f_i(\vy)\|_{\mLambda_j^{-1}}^2\big] \le 
\|\vx - \vy\|_{\mQ^j}^2,
\end{equation}
where $i$ is drawn uniformly at random from $[n].$
\end{assumption}
Similarly, Assumption~\ref{assmpt:vr-Q-hat-tilde} implies Assumption~\ref{assmpt:Q-hat}, so we use the same matrix $\mQ^j$ for both cases. 
Finally, we introduce a useful result on variance bound from~\cite{zheng2016fast} for our later convergence analysis, with the proof provided in Appendix~\ref{appx:prelim} for completeness.
\begin{restatable}{lemma}{finiteVR}\label{lem:batch-vr}
Let $\gB$ be the set of $|\gB| = b$ samples from $[n]$, drawn without replacement and uniformly at random. Then, $\forall \vx \in \rr^d$ and $j \in [m]$,
\begin{equation}\label{eq:vr}
\begin{aligned}
\;& \E_{\mathcal{B}}\Big[\big\|\frac{1}{b}\sum_{i\in\gB}\nabla^j f_i(\vx) - \nabla^j f(\vx)\big\|^2\Big] \\
= \:& \frac{n-b}{b(n-1)}\E_i\big[ \| \nabla^j f_i(\vx) - \nabla^j f(\vx)\|^2 \big].   
\end{aligned}
\end{equation}
\end{restatable}


\section{P-CCD}\label{sec:bccd}

As a warmup, in this section we provide a novel analysis of the standard Proximal Cyclic Block Coordinate Descent (P-CCD) algorithm (Algorithm~\ref{alg:ccd}) for the deterministic setting, adapted to our choice of block norms $\|\cdot\|_{\mLambda_j}.$ P-CCD cycles through the $m$ blocks of variables, updating one block at a time. 
When $m = 1$, P-CCD is the standard proximal gradient method, while when $m = d$, P-CCD is a proximal version of cyclic coordinate descent. 
\begin{algorithm}[h!]
\caption{Proximal Cyclic Block Coordinate Descent (P-CCD)}\label{alg:ccd}
\begin{algorithmic}[1]
\STATE \textbf{Input:} $m, K, \vx_{0}, \mLambda_1, \mLambda_2, \ldots, \mLambda_m$
\FOR{$k = 1$ to $K$} 
\FOR{$j = 1$ to $m$} %
\STATE $\vx_{k-1,  j} = (\vx^{1}_{k}, \ldots, \vx^{j-1}_{k}, \vx^{j}_{k-1},\ldots   \vx^{m}_{k-1})  $
\STATE $\vx_k^j = \argmin_{\vx^j\in\sR^{d_j}}\Big\{\langle \nabla^j f(\vx_{k-1,  j}), \vx^j\rangle + \frac{1}{2}\|\vx^j - \vx_{k-1}^j\|_{\mLambda_j}^2 + r^j(\vx^j)\Big\}$
\ENDFOR
\ENDFOR
 \STATE \textbf{return} $\argmin_{\vx_k}\|\vx_k - \vx_{k-1}\|_{\mLambda} \; (k\in[K])$
\end{algorithmic}	
\end{algorithm}

To analyze the convergence of Algorithm~\ref{alg:ccd}, we first define 
$$\hat{L} := \Big\|\mLambda^{-1/2} \big(\sum_{j=1}^m \hat{\mQ}^j\big)  \mLambda^{-1/2}\Big\|,$$ 
to simplify the notation. This constant appears in the analysis but is not used by the algorithm. 

The analysis is built on two key lemmas.
Lemma~\ref{lem:ccd-grad} bounds the norm of the gradient by the distance between successive iterates, using the generalized Lipschitz condition w.r.t.~a Mahalanobis norm, as stated in Assumption~\ref{assmpt:Q-hat}. Lemma~\ref{lem:ccd-descent} then bounds the sum of successive squared distances between iterates by the initial optimality gap, similar to a result that is typically proved for the (full-gradient) proximal method. Jointly, these two lemmas  lead to a guarantee of a proximal method, but with the generalized Lipschitz constant $\hat{L}$ replacing the traditional full-gradient Lipschitz constant encountered in full-gradient methods.

\begin{restatable}{lemma}{lemmaccdgrad}\label{lem:ccd-grad}
Under Assumption~\ref{assmpt:Q-hat}, the iterates $\{\vx_k\}$ generated by Algorithm~\ref{alg:ccd} satisfy 
\begin{align}\notag
 {\dist}^2(\partial F(\vx_k), \vzero) \le \;&  2(\hat{L} + 1) \|\vx_k - \vx_{k-1}\|^2_{\mLambda}.
\end{align}
\end{restatable}
To bound $\|\vx_k - \vx_{k-1}\|^2_{\mLambda}$, we prove the following descent lemma induced by the block-wise smoothness in Assumption~\ref{assmpt:coordinate-smooth} and by telescoping cyclically over the blocks.
\begin{restatable}{lemma}{lemmaccddescent}\label{lem:ccd-descent}
Under Assumptions~\ref{assmpt:coordinate-smooth}~and~\ref{assmpt:Q-hat}, the iterates $\{\vx_k\}$ generated by Algorithm~\ref{alg:ccd} satisfy
\begin{align}
\sum_{i=1}^k \|\vx_i - \vx_{i-1}\|_{\mLambda }^2\le  2(F(\vx_0)  - F(\vx^*)).
\end{align}
\end{restatable}
Proofs of Lemmas~\ref{lem:ccd-grad}~and~\ref{lem:ccd-descent} are deferred to Appendix~\ref{appx:bccd}. The next result describes the convergence of Algorithm~\ref{alg:ccd}.
\begin{theorem}\label{thm:ccd}
Under Assumptions~\ref{assmpt:coordinate-smooth}~and~\ref{assmpt:Q-hat}, let $\vx^*$ be a global minimizer of \eqref{eq:prob} and $\{\vx_k\}$ be the iterates generated by Algorithm~\ref{alg:ccd}. Then after $K$ (outer-loop) iterations, we have
\begin{align}\notag
\min_{k\in[K]}    {\dist}^2(\partial F(\vx_k), \vzero) 
\le\;& \frac{4(\hat{L} + 1)(F(\vx_{0}) - F(\vx^*))}{K}. 
\end{align}
\end{theorem}

\begin{proof}
By combining Lemmas \ref{lem:ccd-grad} and \ref{lem:ccd-descent}, we have 
\begin{equation*}
\begin{aligned}
\sum_{k=1}^K  {\dist}^2(\partial F(\vx), \vzero)
\leq \;& 2(\hat{L}  + 1) \sum_{k=1}^K \|\vx_k - \vx_{k-1}\|_{\mLambda }^2 \\
\leq \;& 4(\hat{L}  + 1) (F(\vx_0) - F(\vx^*)). 
\end{aligned}
\end{equation*}
It remains to use that for all $k \in [K]$, we have ${\dist}^2(\partial F(\vx_k), \vzero) \geq \min_{k' \in [K]} {\dist}^2(\partial F(\vx_{k'}), \vzero)$.
\end{proof}
In comparison with guarantees with respect to Euclidean Lipschitz constants, this guarantee is never worse than by a factor $m$. 
In the case in which $f$ is $L$-smooth  $\mLambda = \mQ^j = L\mI_d$ (the worst case), we have ${\inf}_{r'(\mathbf{x}) \in \partial r(\mathbf{x})} \|\nabla f(\mathbf{x}) + r'(\mathbf{x})\|^2 = L\textup{dist}^2(\partial F(\mathbf{x}), \bm{0})$ and $\hat{L} = \frac{1}{L}\|\sum_{j = 1}^m \hat{\bm{Q}}^j\| \leq \frac{1}{L}\|\sum_{j = 1}^m \bm{Q}^j\| \leq m$, while in practice possibly $\|\sum_{j = 1}^m \hat{\bm{Q}}^j\| \ll \|\sum_{j = 1}^m \bm{Q}^j\|$ and $\hat{L} \ll m$ (see discussions in e.g., Song and Diakonikolas, 2021). 
The same points hold for guarantees in Section~\ref{sec:vr-bccd} as well. Note that the linear dependence on $m$ cannot be improved in the worst case for standard P-CCD, even on smooth convex problems~\citep{sun2021worst, kamri2022worst}.

If $F$ further satisfies the P{\L} condition of Assumption~\ref{assmpt:PL-condition}, Algorithm~\ref{alg:ccd} can achieve a faster, linear convergence rate. We summarize this result in Corollary~\ref{thm:ccd-pl} below, deferring the proof to Appendix~\ref{appx:bccd}.
\begin{restatable}{corollary}{thmccdpl}\label{thm:ccd-pl}
Suppose that the conditions of Theorem~\ref{thm:ccd} hold and that $F$ further satisfies Assumption~\ref{assmpt:PL-condition}. 
Then we have after $K$ iterations of Algorithm~\ref{alg:ccd} that
\begin{equation}\notag
    F(\vx_K) - F(\vx^*) \leq \Big(\frac{2(\hat{L} + 1)}{2(\hat{L}  + 1) + \mu}\Big)^{K}(F(\vx_{0}) - F(\vx^*)).
\end{equation}
\end{restatable}
The main bottleneck in implementing P-CCD is in finding appropriate matrices $\mLambda_j$ that satisfy Assumption~\ref{assmpt:coordinate-smooth}. 
The simplest approach is to use $\mLambda_j = L_j \mI_{d_j}$ and estimate $L_j$ adaptively using the standard backtracking line search. 
This procedure can be implemented efficiently, as the analysis requires Assumption~\ref{assmpt:coordinate-smooth} to hold only between successive iterates. The use of more general diagonal matrices $\mLambda_j$ 
is a form of block preconditioning, which is frequently used to heuristically improve the performance of full-gradient methods.
In our neural net training experiments, for example, we use spectral normalization (see Section~\ref{sec:num-exp} for more details). 

\section{Variance Reduced P-CCD}\label{sec:vr-bccd}
We now consider nonconvex optimization problems of the form \eqref{eq:sto-opt}. 
When $n$ is finite, \eqref{eq:sto-opt} is a finite-sum problem and we can compute the full gradient of $F(\vx)$ with $O(nd)$ cost. Without loss of generality, we use $n = +\infty$ to denote the general stochastic optimization setting as in \eqref{pro:stoc}, where the full gradient can no longer be computed in finite time. For both settings, Algorithm~\ref{alg:vr-ccd} describes VR-CCD, the Variance-Reduced Cyclic block Coordinate Descent algorithm, which combines Algorithm~\ref{alg:ccd} with recursive variance reduction of PAGE type~\citep{li2020page} to reduce the per-iteration cost and improve the overall complexity.

\begin{algorithm*}[h!]
\caption{Variance-Reduced Cyclic Block Coordinate Descent (VR-CCD)}\label{alg:vr-ccd}
\begin{algorithmic}[1]
\STATE \textbf{Input:} $m$, $\eta$, $K$, $p$, $b$, $b'$, $\mLambda_1, \mLambda_2, \ldots, \mLambda_m$, $\vx_0 = \vx_{-1} = \vx_{-1, 1} = \cdots = \vx_{-1, m + 1}$.
\STATE $ \vg_{-1} = \frac{1}{b}\sum_{i\in\gB} \nabla f_i(\vx_{0})$. 
\FOR{$k = 1$ to $K$} 
\FOR{$j = 1$ to $m$} %
\STATE $\vx_{k-1,  j} = (\vx^{1}_{k}, \ldots, \vx^{j-1}_{k}, \vx^{j}_{k-1},\ldots   \vx^{m}_{k-1})$
\STATE $\vg_{k-1}^j = 
\begin{cases}
\frac{1}{b}\sum_{i\in\gB} \nabla^j f_i(\vx_{k-1,  j}), \text{ with probability } p \\ 
\vg_{k-2}^j +\frac{1}{b'} \sum_{i\in\gB'} (\nabla^j f_i(\vx_{k-1,  j}) - \nabla^j f_i(\vx_{k-2, j})),     \text{ with probability } 1-p
\end{cases}$
\STATE $\vx_k^j = \argmin_{\vx^j\in\sR^d}\Big\{\langle \vg_{k-1}^j, \vx^j\rangle + r^j(\vx^j) +  \frac{1}{2\eta}\|\vx^j -  \vx_{k-1}^j\|_{\mLambda_j}^2\Big\}$
\ENDFOR
\ENDFOR
\STATE \textbf{return} $\hat{\vx}_K$ uniformly drawn from $\{\vx_k\}_{k \in [K]}$
\end{algorithmic}	
\end{algorithm*}

Instead of computing the block-wise gradient at each inner iteration as P-CCD (Algorithm~\ref{alg:ccd}), 
VR-CCD maintains and updates a recursive gradient estimator $g_{k - 1}^j$ 
of PAGE type for each block gradient $j$ at outer iteration $k$ (i.e., $g_{k - 1}^j$ estimates $\nabla^j f(\vx_{k - 1, j})$). 
By the definition of $g_{k - 1}^j$ in Line~$6$ of Algorithm~\ref{alg:vr-ccd}, it uses a mini-batch estimate $\frac{1}{b}\sum_{i\in\gB} \nabla^j f_i(\vx_{k-1,  j})$ with probability $p$, where $|\gB| = b$. 
With probability $1 - p$, the estimate $\vg_{k-2}^j +\frac{1}{b'} \sum_{i\in\gB'} (\nabla^j f_i(\vx_{k-1,  j}) - \nabla^j f_i(\vx_{k-2, j}))$ reuses the previous $j$\textsuperscript{th} block gradient estimator $g_{k - 2}^j$, and forms an approximation of the gradient difference $\nabla^j f(\vx_{k - 1, j}) - \nabla^j f(\vx_{k - 2, j})$ based on the minibatch $\gB'$ where $|\gB'| = b'$. 
When  $p = 1$, the PAGE estimator reduces to vanilla minibatch SGD. To lower the computational cost, it is common to take $b' \ll b$ and $p \ll 1$. 
The estimator $g_{k - 1}^j$ is then incorporated into the Lipschitz gradient surrogate function in Line~$7$ to compute the new iterate $\vx_{k}^j$. 
The variance of PAGE estimator w.r.t.~block coordinates can be bounded recursively as in Lemma~\ref{lem:vr} below, using the minibatch variance bound results in Lemma~\ref{lem:batch-vr}. The proof appears in Appendix~\ref{appx:vr-bccd}.

To simplify the notation, we use the following definitions in the statements and proofs for this section,  for $k \geq 0$:
\begin{align*}
    \tilde{L} & := \big\|\mLambda^{-1/2} \big(\sum_{j=1}^m\tilde{\mQ}^j\big) \mLambda^{-1/2}\big\|, \\
    u_k & := \sum_{j=1}^m\|\vg_{k-1}^j - \nabla^j  f(\vx_{k-1,  j})\|_{\mLambda_j^{-1}}^2, \\
    v_k & := \|\vx_k -  \vx_{k-1}\|_{\mLambda}^2, \\
    s_k & :=  {\dist}^2(\partial F(\vx_k), \vzero).
\end{align*}
\begin{restatable}{lemma}{lemmavr}\label{lem:vr}
Suppose Assumptions~\ref{assmpt:finite-var}--\ref{assmpt:vr-Q-hat-tilde} hold, then the variance $\E[u_k]$ of the gradient estimators $\{\vg_{k-1}^j\}_{j = 1}^m$ at iteration $k$ of Algorithm~\ref{alg:vr-ccd} is bounded by:
\begin{equation}\label{eq:vr-bound}
\begin{aligned}
\E[u_k] \leq \;& \frac{2p(n-b)\sigma^2}{b(n-1)} + 2\Big(\frac{p(n-b)}{b(n-1)}+\frac{1-p}{b'}\Big)\tilde{L}\E[v_k]\notag \\
& + (1-p)\E[u_{k - 1}] + \frac{2(1-p)\hat{L}}{b'} \E[v_{k - 1}].
\end{aligned}
\end{equation}
\end{restatable}
The general strategy to analyze the convergence of VR-CCD can be summarized as follows. Let $$\Phi_k = a_k\ee[F(\vx_k)] + b_k\ee[u_k] + c_k\ee[v_k]$$
be a potential function, where $\{a_k\}_{k \geq 1}, \{b_k\}_{k \geq 1}, \{c_k\}_{k \geq 1}$ are non-negative sequences to be specified later in the analysis. Our goal is to show that 
\begin{equation}\label{eq:pot-fn-telescope}
    \ee[s_k] \leq \Phi_{k - 1} - \Phi_k + \mathcal{E}_k,
\end{equation} 
where $\mathcal{E}_k$ are error terms arising from the noise of the estimator. Then, by telescoping \eqref{eq:pot-fn-telescope} and controlling the error sequence $\mathcal{E}_k$, we obtain the gradient norm guarantee of Algorithm~\ref{alg:vr-ccd}. First, we make use of the following descent lemma that utilizes block-wise smoothness from  Assumption~\ref{assmpt:coordinate-smooth-sto}. Its proof is deferred to Appendix~\ref{appx:vr-bccd}.

\begin{restatable}{lemma}{lemmavrdescent}\label{lem:vr-descent}
Let Assumption~\ref{assmpt:coordinate-smooth-sto} hold and let $r^{j, '}(\vx_{k}^j) \in \partial r^j(\vx_{k}^j)$ be such that $\vx_k^j = \vx_{k-1}^j - \eta\mLambda_j^{-1}(\vg_{k-1}^j + r^{j, '}(\vx_k^j)).$ Then the iterates of Algorithm~\ref{alg:vr-ccd} satisfy  
\begin{equation}\label{eq:descent-vr}
\begin{aligned}
  F(\vx_{k}) \le \;& F(\vx_{k-1}) -\frac{1-\eta}{2\eta}v_k + \frac{\eta}{2}u_k \\
  & - \frac{\eta}{2} \sum_{j = 1}^{m}\|\nabla^j f(\vx_{k-1,  j}) + r^{j, '}(\vx_k^j)\|_{\mLambda_j^{-1}}^2. 
\end{aligned}
\end{equation}
\end{restatable}
In the statement of Lemma~\ref{lem:vr-descent}, there must exist $r^{j, '}(\vx_{k}^j) \in \partial r^j(\vx_{k}^j)$  such that $\vx_k^j = \vx_{k-1}^j - \eta\mLambda_j^{-1}(\vg_{k-1}^j + r^{j, '}(\vx_k^j)),$ due to the first-order optimality condition of the minimization problem that defines $\vx_k^j.$

We further bound the gradient norms of intermediate iterates within a cycle in Inequality~\eqref{eq:descent-vr}, i.e. $\sum_{j = 1}^{m}\|\nabla^j f(\vx_{k-1,  j}) + r^{j, '}(\vx_k^j)\|_{\mLambda_j^{-1}}^2$, using smoothness from Assumption~\ref{assmpt:vr-Q-hat-tilde}. 
\begin{restatable}{lemma}{lemmavrgradnorm}\label{lem:vr-grad-norm}
Let Assumption~\ref{assmpt:vr-Q-hat-tilde} hold and let $r^{j, '}(\vx_{k}^j) \in \partial r^j(\vx_{k}^j)$ be such that $\vx_k^j = \vx_{k-1}^j - \eta\mLambda_j^{-1}(\vg_{k-1}^j + r^{j, '}(\vx_k^j)).$ Then for Algorithm~\ref{alg:vr-ccd} we have 
\begin{equation}
\begin{aligned}
s_k \le 2 \hat{L} v_k + 2\sum_{j=1}^m  \|\nabla^j f(\vx_{k-1,  j}) + r^{j,'}(\vx_k) \|_{\mLambda_j^{-1}}^2. 
\end{aligned}
\end{equation}
\end{restatable}
By combining Lemma~\ref{lem:vr-descent}~and~\ref{lem:vr-grad-norm} and using the recursive variance bound of the estimator from Lemma~\ref{lem:vr}, we are ready to prove a bound on iteration complexity for VR-CCD in Theorem~\ref{thm:vr-main-result}. The proof is in Appendix~\ref{appx:vr-bccd}. 
\begin{restatable}{theorem}{theoremvrmainresult}\label{thm:vr-main-result}
Suppose that 
Assumptions~\ref{assmpt:coordinate-smooth}--\ref{assmpt:Q-hat}~and~\ref{assmpt:finite-var}--\ref{assmpt:vr-Q-hat-tilde}
hold. Let $\vx^*$ be a global minimizer of \eqref{eq:prob} and $\{\vx_k\}$ be the iterates generated by Algorithm~\ref{alg:vr-ccd}. Then, we have
\begin{equation}\label{ineq:rate}
\begin{aligned}
\;& \E\Big[{\dist}^2(\partial F(\hat{\vx}_K), \vzero)\Big] \\
\leq \;& \frac{4\Delta_0}{\eta K} + \frac{2(1-p)(n-b)\sigma^2}{pb(n-1) K} + \frac{4(n-b)\sigma^2}{b(n-1)},
\end{aligned}
\end{equation}
where $\Delta_0 = F(\vx_0) - F(\vx^*)$, and $0<\eta \le \frac{-1 + \sqrt{1+4c_0}}{2c_0}$ with $c_0 = \frac{2(1-p)\hat{L}}{pb'} + \hat{L} + 2\Big( \frac{p(n-b)}{b(n-1)}+\frac{1-p}{b'}  \Big)\frac{\tilde{L}}{p}.$ 
\end{restatable}
Note that Theorem~\ref{thm:vr-main-result} is generic for both finite-sum and infinite-sum cases. We summarize its implications for both problems in the following corollaries, for specific parameters of Algorithm~\ref{alg:vr-ccd}.
In the remaining results of this section, we assume that the assumptions of Theorem~\ref{thm:vr-main-result} hold.
The proofs are provided in Appendix~\ref{appx:vr-bccd} for completeness.

\begin{restatable}[Finite-sum]{corollary}{finiteComplexity}\label{cor:finite-sum}
Choosing $b = n$, $b' = \sqrt{n}$, and $p = \frac{b'}{b + b'}$, and setting $K = \frac{4(F(\vx_0) - F(\vx^*))}{\epsilon^2\eta}$, we have 
$
\E\Big[{\dist}^2(\partial F(\hat{\vx}_K), \vzero)\Big] \leq \epsilon^2 
$
with $\mathcal{O}\big(nd + \frac{\Delta_0d\sqrt{n(\hat{L} + \tilde{L})}}{\epsilon^2}\big)$ arithmetic operations, where $\Delta_0 = F(\vx_0) - F(\vx^*)$.
\end{restatable}
\begin{restatable}[Infinite-sum]{corollary}{onlineComplexity}\label{cor:online}
Choosing $b = \min\{\lceil\frac{12\sigma^2}{\epsilon^2}\rceil, n\}$, $b' = \sqrt{b}$, and $p = \frac{b'}{b + b'}$, and setting $K = \frac{12(F(\vx_0) - F(\vx^*))}{\epsilon^2\eta} + \frac{1}{2p}$, we have 
$
\E\Big[{\dist}^2(\partial F(\hat{\vx}_K), \vzero)\Big] \leq \epsilon^2
$
with $\mathcal{O}\big(bd + \frac{\Delta_0d\sqrt{b(\hat{L} + \tilde{L})}}{\epsilon^2}\big)$ arithmetic operations, where $\Delta_0 = F(\vx_0) - F(\vx^*)$.
\end{restatable}
Under the P{\L} condition, a faster convergence rate can be proved for VR-CCD, as we show next.
\begin{restatable}{corollary}{thmccdvrpl}\label{thm:ccd-vr-pl}
Suppose that F further satisfies Assumption~\ref{assmpt:PL-condition}, then we have for
Algorithm~\ref{alg:vr-ccd},
\begin{equation*}
\begin{aligned}
    \;& \ee[ F(\vx_K) - F(\vx^*)] \\
    \leq \;& \big(1 + \frac{\eta\mu}{2}\big)^{-K}\Big(\Delta_0 + \frac{\sigma^2\eta(1 - p)(n - b)}{pb(n - 1)}\Big) + \frac{4(n - b)\sigma^2}{b\mu(n - 1)},
\end{aligned}
\end{equation*}
where $0 < \eta \leq \min\Big\{\frac{p}{\mu(1 - p)}, \frac{-1 + \sqrt{1 + 4c_0}}{2c_0}\Big\}$ with $c_0 = \hat{L} + \frac{4\hat{L}}{pb'} + \frac{4\tilde{L}}{p}\big(\frac{p(n - b)}{b(n - 1)} + \frac{1 - p}{b'}\big)$, and $\Delta_0 = F(\vx_0) - F(\vx^*)$.
\end{restatable}
In the following, we further provide arithmetic operation complexity results under P{\L} condition with specifying the parameters of the algorithm for both finite-sum and infinite-sum problems.
\begin{restatable}[Finite-sum]{corollary}{finiteComplexityPL}
Choosing $b = n$, $b' = \sqrt{n}$, and $p = \frac{b'}{b + b'}$, and setting $K = \big(1 + \frac{2}{\eta\mu}\big)\log\big(\frac{\Delta_0}{\epsilon}\big)$, we have 
$
\E[F(\vx_K) - F(\vx^*)] \leq \epsilon
$
with $\mathcal{O}\Big(nd + \big(\frac{\sqrt{n(\hat{L} + \tilde{L})}}{\mu} + n\big)d\log\big(\frac{\Delta_0}{\epsilon}\big)\Big)$ arithmetic operations, where $\Delta_0 = F(\vx_0) - F(\vx^*)$.
\end{restatable}
\begin{restatable}[Infinite-sum]{corollary}{onlineComplexityPL}
Choosing $b = \min\{\lceil\frac{12\sigma^2}{\mu\epsilon}\rceil, n\}$, $b' = \sqrt{b}$, and $p = \frac{b'}{b + b'}$, and setting $K = \big(1 + \frac{2}{\eta\mu}\big)\log\big(\frac{3\Delta_0}{\epsilon}\big)$, we have 
$
\E[F(\vx_K) - F(\vx^*)] \leq \epsilon
$
with $\mathcal{O}\Big(bd + \big(\frac{\sqrt{b(\hat{L} + \tilde{L})}}{\mu} + b\big)d\log\big(\frac{\Delta_0}{\epsilon}\big)\Big)$ arithmetic operations, where $\Delta_0 = F(\vx_0) - F(\vx^*)$.
\end{restatable}

A few remarks are in order here. 
In addition to requiring matrices $\mLambda_j,$ VR-CCD also requires constants $\hat{L}$ and $\tilde{L}$ to set the learning rate, but these constants are often not readily available in practice. Instead, one can tune the value of the learning rate $\eta,$ as is frequently done in practice for other optimization methods. 
Additionally, our methods require a fresh sample for each block in the cyclic update --- an undesirable feature, because the sample complexity increases with the number of blocks. 
However, this requirement appears only to be an artifact of the analysis. In practice, we can implement a variant of VR-CCD (VRO-CCD) which re-uses the same sample for all the blocks. 
Interestingly, this variant shows even better empirical performance than VR-CCD (see Fig.~\ref{fig:vr-compare} in Appendix~\ref{appx:exp}). Analysis of this variant is beyond the scope of this paper and is an interesting direction for future research. 

\section{Numerical Experiments and Discussion}\label{sec:num-exp}
\begin{figure*}[ht]
    \hspace*{\fill}\subfigure[Train Loss]{\includegraphics[width=0.25\textwidth]{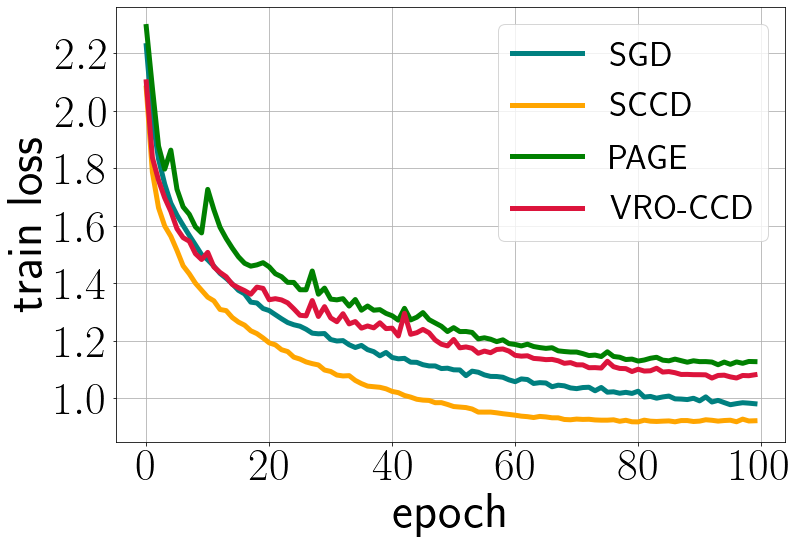}\label{fig:train-loss}}\hfill
    \subfigure[Test Accuracy]{\includegraphics[width=0.25\textwidth]{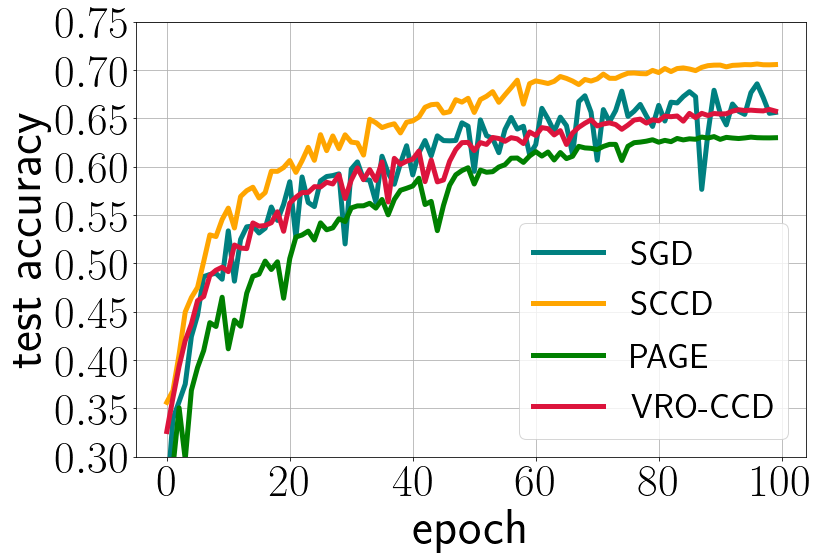}\label{fig:test-acc}}\hfill
    \subfigure[Train Loss (SN)]{\includegraphics[width=0.25\textwidth]{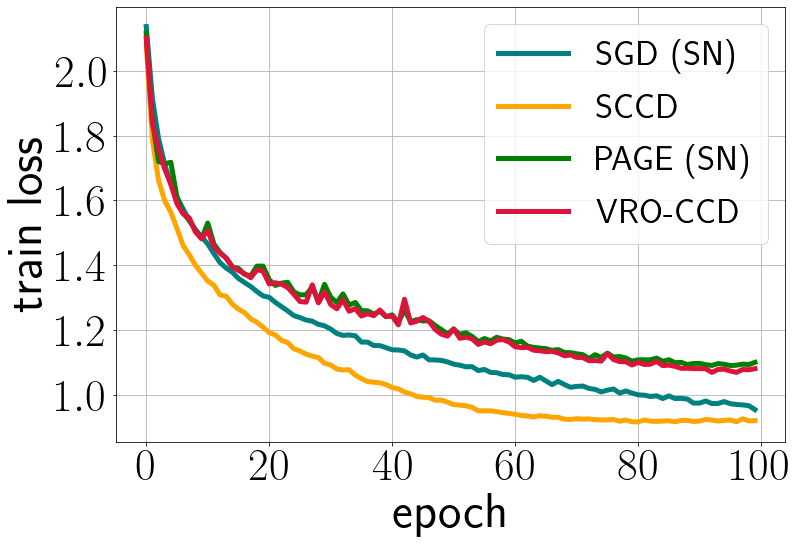}\label{fig:train-loss-sn}}\hfill
    \subfigure[Test Accuracy (SN)]{\includegraphics[width=0.25\textwidth]{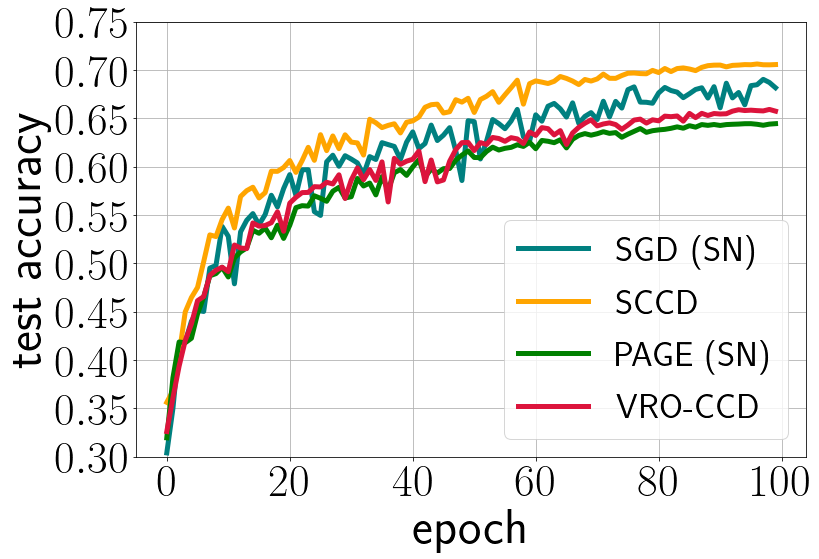}\label{fig:test-acc-sn}}
    \hspace*{\fill}
    \caption{Comparison of SGD, SCCD, PAGE and VRO-CCD on training LeNet on CIFAR-10.
    }
    \label{fig:cifar}
\end{figure*}
We now describe numerical experiments on training neural networks to evaluate the cyclic block coordinate update scheme. In particular, we train LeNet~\citep{lecun1998gradient} with weight decay for the image classification task on CIFAR-10~\citep{krizhevsky2009learning} for our experiments.
Details of the architecture are provided in Appendix~\ref{appx:exp}. 
We implement VRO-CCD and its special case SCCD (obtained by setting $p = 1$), and compare them with SGD and PAGE~\citep{li2020page}. 
Note that VRO-CCD and PAGE, SCCD and SGD  differ only in whether the variables are cyclically updated or not, so provide a fair comparison to justify the efficacy of the cyclic scheme. We implement all the algorithms using PyTorch~\citep{paszke2019pytorch}, and run the experiments on Google Colab standard GPU backend.

For all algorithms, we set the mini-batch size $b$ to be $512$ and the weight decay parameter to be $0.0005$. We repeat the experiments 3 times with 100 epochs and average the results. For the learning rate, we use the cosine learning rate scheduler~\citep{loshchilov2016sgdr}, which is tuned  separately for each method via a grid search. For VRO-CCD and PAGE methods, we set  $b' = \sqrt{b}$ and $p = \frac{b'}{b + b'}$ according to the theoretical results. For VRO-CCD and SCCD, we split the neural network parameters with each layer as a block ($m = 5$), and estimate the $\mLambda_j$ of fully connected layers by spectral norms of the weights using spectral normalization (SN)~\citep{miyato2018spectral}. VRO-CCD and SCCD are also compared with SGD and PAGE with the same spectral normalization. We report and plot the train loss and test accuracy against the epoch numbers in Fig.~\ref{fig:cifar}, where one epoch corresponds to the number of arithmetic operations in one data pass, in the order $\mathcal{O}(Nd)$ where $N$ is the size of the training set. We summarize the runtime of each algorithm per iteration\footnote{One iteration for cyclic methods refers to one cycle of block coordinate updates. Since PAGE switches between large and small batches, we only summarize the mean runtime per iteration here.} and per epoch in Table~\ref{table:time}, which is averaged by the results of $100$ epochs. 

\begin{table}[ht]
\caption{Runtime (seconds) per iteration and per epoch.}
\label{table:time}
\vskip 0.15in
\begin{center}
\begin{small}
\begin{sc}
\begin{tabular}{lcccr}
\toprule
Algorithm & Runtime (Iter) & Runtime (Epoch) \\
\midrule
SGD     & 0.172 & 16.9 $\pm$ 0.9 \\
SCCD    & 0.185 & 18.1 $\pm$ 0.6 \\
PAGE    & 0.017 & 20.7 $\pm$ 0.9 \\
VRO-CCD & 0.041 & 49.5 $\pm$ 2.3 \\
\bottomrule
\end{tabular}
\end{sc}
\end{small}
\end{center}
\vskip -0.1in
\end{table}
From Fig.~\ref{fig:cifar}, we observe that (i) SCCD and VRO-CCD with cyclic scheme exhibit faster convergence with better generalization than SGD and PAGE, respectively, in Figures~\ref{fig:train-loss}~and~\ref{fig:test-acc}; (ii) The edge of cyclic scheme is still noticeable comparing to SGD and PAGE with the same spectral normalization in Figures~\ref{fig:train-loss-sn}~and~\ref{fig:test-acc-sn}. All of these validate the efficacy of the cyclic update scheme. Note that in this experiment SGD and SCCD can admit larger stepsize, thus showing faster convergence (see Fig.~\ref{fig:cifar-same} in Appendix~\ref{appx:exp} for further numerical comparison with same stepsizes).  

When using small batches, VRO-CCD is around $2.4$ times slower than PAGE as in Table~\ref{table:time}, because the cyclic update becomes the major computational bottleneck in each iteration.
However, for large batches, SCCD  sacrifices only a marginal $7.5\%$ runtime per iteration in comparison with SGD. SCCD also converges fastest in terms of wall-clock time (see Fig.~\ref{fig:cifar-time} in Appendix~\ref{appx:exp}, as one may worry about the accumulation of marginal time increase over epochs). 
Our conclusion is that the cyclic scheme can be an efficient and effective alternative for large-batch methods.

\section*{Acknowledgements}

XC and CS acknowledge support from NSF DMS 2023239. 
SW acknowledges support from NSF DMS 2023239 and CCF 2224213, AFOSR under subcontract UTA20-001224 from University of Texas-Austin, DOE under subcontract 8F-30039 from Argonne National Laboratory. JD acknowledges support from the U.S.~Office of Naval Research under contract number N000142212348 and from NSF Award CCF 2007757.

\bibliographystyle{plainnat}
\bibliography{reference}

\newpage
\onecolumn
\appendix
\section{Omitted Proofs from Section~\ref{sec:prelim}}\label{appx:prelim}
\finiteVR*
\begin{proof}
Observe first by expanding the square that 
\begin{equation*}
    \begin{aligned}
        \;& \E_{\mathcal{B}}\Big[\big\|\frac{1}{b}\sum_{i\in\gB}\nabla^j f_i(\vx) - \nabla^j f(\vx)\big\|^2\Big] \\
        = \;& \frac{1}{b^2}\E_{\mathcal{B}}\Big[\sum_{i, i' \in \mathcal{B}}\innp{\nabla^j f_i(\vx) - \nabla^j f(\vx), \nabla^j f_{i'}(\vx) - \nabla^j f(\vx)}\Big] \\
        = \;& \frac{1}{b^2}\E_{\mathcal{B}}\Big[\sum_{i \neq i' \in \mathcal{B}}\innp{\nabla^j f_i(\vx) - \nabla^j f(\vx), \nabla^j f_{i'}(\vx) - \nabla^j f(\vx)}\Big] + \frac{1}{b}\E_i\Big[\|\nabla^j f_i(\vx) - \nabla^j f(\vx)\|^2\Big]. 
    \end{aligned}
\end{equation*}
Since the batch $\mathcal{B}$ is drawn independently and uniformly from $[n]$, we know the probability that each pair $(i, i')$ with $i \neq i'$ belongs to $\mathcal{B}$ can be given as $\frac{b(b - 1)}{n(n - 1)}$. Further, by the linearity of expectation, we have 
\begin{equation*}
    \begin{aligned}
        \;& \E_{\mathcal{B}}\Big[\sum_{i \neq i' \in \mathcal{B}}\innp{\nabla^j f_i(\vx) - \nabla^j f(\vx), \nabla^j f_{i'}(\vx) - \nabla^j f(\vx)}\Big] \\
        = \;& \E_{\mathcal{B}}\Big[\sum_{i \neq i' \in [n]}\mathbbm{1}_{i \neq i' \in \mathcal{B}}\innp{\nabla^j f_i(\vx) - \nabla^j f(\vx), \nabla^j f_{i'}(\vx) - \nabla^j f(\vx)}\Big]\\
        = \;& \sum_{i \neq i' \in [n]}\E_{\mathcal{B}}\Big[\mathbbm{1}_{i \neq i' \in \mathcal{B}}\innp{\nabla^j f_i(\vx) - \nabla^j f(\vx), \nabla^j f_{i'}(\vx) - \nabla^j f(\vx)}\Big]\\
        = \;& \frac{b(b - 1)}{n(n - 1)}\sum_{i \neq i' \in [n]}\innp{\nabla^j f_i(\vx) - \nabla^j f(\vx), \nabla^j f_{i'}(\vx) - \nabla^j f(\vx)}, 
    \end{aligned}
\end{equation*}
where $\mathbbm{1}$ is the indicator function with $\mathbbm{1}_{i \neq i' \in \mathcal{B}} = 1$ if $i \neq i' \in \mathcal{B}$ otherwise $0$. 
Then we obtain 
\begin{equation*}
    \begin{aligned}
        \;& \E_{\mathcal{B}}\Big[\big\|\frac{1}{b}\sum_{i\in\gB}\nabla^j f_i(\vx) - \nabla^j f(\vx)\big\|^2\Big] \\
        = \;& \frac{b - 1}{bn(n - 1)}\sum_{i \neq i' \in [n]}\innp{\nabla^j f_i(\vx) - \nabla^j f(\vx), \nabla^j f_{i'}(\vx) - \nabla^j f(\vx)} + \frac{1}{b}\E_i\Big[\|\nabla^j f_i(\vx) - \nabla^j f(\vx)\|^2\Big]\\
        = \;& \frac{b - 1}{bn(n - 1)}\sum_{i, i' \in [n]}\innp{\nabla^j f_i(\vx) - \nabla^j f(\vx), \nabla^j f_{i'}(\vx) - \nabla^j f(\vx)} + \Big(\frac{1}{b} - \frac{b - 1}{b(n - 1)}\Big)\E_i\Big[\|\nabla^j f_i(\vx) - \nabla^j f(\vx)\|^2\Big]\\
        \overset{(\romannumeral1)}{=} \;& \frac{n - b}{b(n - 1)}\E_i\Big[\|\nabla^j f_i(\vx) - \nabla^j f(\vx)\|^2\Big], 
    \end{aligned}
\end{equation*}
where $(\romannumeral1)$ is due to the finite-sum structure of $f$, i.e. $f = \frac{1}{n}\sum_{i = 1}^n f_i$. 
\end{proof}
\begin{remark}
The proof for Lemma~\ref{lem:batch-vr} does not involve the specification of the norm, thus applying to $\|\cdot\|_{\mLambda_j^{-1}}$ in the paper.
\end{remark}

\section{Omitted Proofs from Section~\ref{sec:bccd}}\label{appx:bccd}
\lemmaccdgrad*
\begin{proof}
By the definition of $\vx_k^j$ (Step 5 in Algorithm~\ref{alg:ccd}), for each $j \in [m],$ we have that there exists  $r^{j, '}(\vx_k^j) \in \partial r^j(\vx_k^j)$ such that $\vx_k^j = \vx_{k-1}^j - \mLambda_j^{-1}(\nabla^j f(\vx_{k-1}, j) + r^{j, '}(\vx_k^j)).$ Observe that, due to the block separability of $r,$ the vector $r'$ obtained by concatenating such vectors $r^{j, '},$ $j \in [m]$, is a subgradient vector for $r.$  Hence, using block separability of $r$ and  the definitions of the matrix $\mLambda$ and the norm $\|\cdot\|_{\mLambda}$, we have
\begin{align}
\|\nabla f(\vx_k) + r'(\vx_k)\|_{\mLambda^{-1}}^2 = \;&  \sum_{j=1}^m    \|\nabla^j f(\vx_k) + {r^{j, '}}(\vx_k^j)\|_{\mLambda^{-1}_j}^2            \nonumber\\
\le \;& 2 \sum_{j=1}^m   \big( \|\nabla^j f(\vx_k) - \nabla^j f(\vx_{k-1,  j})\|_{\mLambda^{-1}_j}^2 +  \|\nabla^j f(\vx_{k-1,  j}) + {r^{j, '}}(\vx_k^j)\|_{\mLambda_j^{-1}}^2\big)      \nonumber\\
=  \;&  2 \sum_{j=1}^m \big( \|\nabla^j f(\vx_k) - \nabla^j f(\vx_{k-1,  j})\|_{\mLambda^{-1}_j}^2   + \|\mLambda_j  (\vx_k^j - \vx_{k-1}^j)\|_{\mLambda_j^{-1}}^2 \big),
\label{eq:ccd-2}
\end{align}
where the inequality comes from adding and subtracting $\nabla^j f(\vx_{k-1,  j})$ inside the norm terms and using Young's inequality, while the last equality is by $\vx_k^j = \vx_{k-1}^j - \mLambda_j^{-1}(\nabla^j f(\vx_{k-1}, j) + r^{j, '}(\vx_k^j)).$

By Assumption \ref{assmpt:Q-hat}, we have 
\begin{align}
\|\nabla^j f(\vx_k) - \nabla^j f(\vx_{k-1,  j})\|_{\mLambda^{-1}_j}^2 \le (\vx_k -  \vx_{k-1,  j})^T \mQ^j (\vx_k -  \vx_{k-1,  j}),    \label{eq:ccd-grad-Lip}
\end{align}
while using the definition of the norm $\|\cdot\|_{\mLambda_j^{-1}}$, we have
\begin{align}
 \|\mLambda_j (\vx_k^j - \vx_{k-1}^j)\|_{\mLambda_j^{-1}}^2 = \|\vx_k^j - \vx_{k-1}^j\|_{\mLambda_j}^2.    \label{eq:ccd-dist-simplification}
\end{align}

Observe further that $\vx_{k-1,  j}$ has the same elements as $\vx_k$ in the first $j-1$ coordinates. Meanwhile, for $t \ge j,$ let $(\vx_{k-1,  j})_t, (\vx_{k-1})_t$ denote the $t$\textsuperscript{th} blocks of $\vx_{k-1,  j}$ and $\vx_{k-1}$  respectively. Then we have $(\vx_{k-1,  j})_t = (\vx_{k-1})_t,$ thus $(\vx_{k-1,  j})_t - (\vx_{k})_t = (\vx_{k-1})_t - (\vx_{k})_t.$ Thus, based on the definition of $\hat{\mQ}^j$, we have 
\begin{align}
 (\vx_k -  \vx_{k-1,  j})^T \mQ^j (\vx_k -  \vx_{k-1,  j}) = (\vx_k -  \vx_{k-1})^T \hat{\mQ}^j (\vx_k -  \vx_{k-1}).  \label{eq:ccd-Qhat}  
\end{align}

As a result, combining \eqref{eq:ccd-2}--\eqref{eq:ccd-Qhat}, we have
\begin{align}
& \|\nabla f(\vx_k) + r'(\vx_k)\|_{\mLambda^{-1}}^2 \nonumber\\
 \le \;& 2 \sum_{j=1}^m \Big(  (\vx_k - \vx_{k-1})^T \hat{\mQ}^j( \vx_k - \vx_{k-1})  +   \|\vx_k^j - \vx_{k-1}^j\|_{\mLambda_j}^2 \Big)   \nonumber\\
= \;&  2 (\mLambda^{1/2}(\vx_k - \vx_{k-1}))^T \mLambda^{-1/2} \Big( \sum_{j=1}^m  \hat{\mQ}^j\Big)\mLambda^{-1/2}(\mLambda^{1/2}(\vx_k - \vx_{k-1})) 
+ 2\|\vx_k - \vx_{k-1}\|^2_{\mLambda}    \nonumber\\
\leq \;& 2 \Big\| \mLambda^{-1/2} \Big(\sum_{j=1}^m  \hat{\mQ}^j \Big) \mLambda^{-1/2} \Big\|  \|\vx_k - \vx_{k-1}\|^2_{\mLambda}    + 2\|\vx_k - \vx_{k-1}\|^2_{\mLambda}              \nonumber\\
= \;&  2\Big(  \Big\| \mLambda^{-1/2} \Big(\sum_{j=1}^m \hat{\mQ}^j\Big) \mLambda^{-1/2} \Big\|  + 1 \Big) \|\vx_k - \vx_{k-1}\|^2_{\mLambda}. \notag
\end{align}
To complete the proof, it remains to note that, by definition, $\hat{L} = \Big\| \mLambda^{-1/2} \Big(\sum_{j=1}^m \hat{\mQ}^j\Big) \mLambda^{-1/2} \Big\|$ and observe that ${\dist}^2 (\partial F(\vx_k), \vzero) \leq \|\nabla f(\vx_k) + r'(\vx_k)\|_{\mLambda^{-1}},$ for any $r'(\vx_k) \in \partial r(\vx_k).$ 
\end{proof}

\lemmaccddescent*
\begin{proof}
Let $r^{j, '}(\vx_k^j) \in \partial r^j(\vx_k^j)$ be such that $\vx_k^j = \vx_{k-1}^j - \mLambda_j^{-1}(\nabla^j f(\vx_{k-1}, j) + r^{j, '}(\vx_k^j)).$ By Assumption \ref{assmpt:coordinate-smooth} and the definition of $\vx_{k-1, j+1}$, we have 
\begin{align}
F(\vx_{k-1,  j+1}) = \;& f(\vx_{k-1,  j+1}) + r(\vx_{k-1,  j+1}) \nonumber  \\
\le\;&  f(\vx_{k-1,  j}) + \langle \nabla^j  f(\vx_{k-1,  j}), \vx_k^j - \vx_{k-1}^j\rangle + \frac{1}{2}\|\vx_k^j - \vx_{k-1}^j\|_{\mLambda_j}^2   \nonumber  \\
\;& + \sum_{i=1}^{j-1} r^i(\vx_k^i)  + r^j(\vx_k^j) + \sum_{i=j+1}^{m} r^i(\vx_{k-1}^i)   \nonumber  \\  
\le\;&  f(\vx_{k-1,  j}) + r(\vx_{k-1,  j}) + \langle \nabla^j f(\vx_{k-1,  j}) + r^{j,'}(\vx_{k}^j), \vx_k^j - \vx_{k-1}^j\rangle +\frac{1}{2}\|\vx_k^j - \vx_{k-1}^j\|_{\mLambda_j}^2   \nonumber  \\  
=\;&  F(\vx_{k-1,  j}) - \frac{1}{2}\|\vx_k^j - \vx_{k-1}^j\|_{\mLambda_j}^2,   \label{eq:descent}
\end{align}
where the second inequality is by the definition of a subgradient and last equality is by $\vx_k^j = \vx_{k-1}^j - \mLambda_j^{-1}(\nabla^j f(\vx_{k-1}, j) + r^{j, '}(\vx_k^j)).$  
Applying~\eqref{eq:descent} recursively over $j=1,2,\dotsc,m$, we have 
\begin{align}
F(\vx_k) =   F(\vx_{k-1,m+1}) \le\;&  F(\vx_{k-1,1}) - \sum_{j=1}^m \frac{1}{2}\|\vx_k^j - \vx_{k-1}^j\|_{\mLambda_j}^2 \nonumber\\
=\;& F(\vx_{k-1}) -\frac{1}{2} \|\vx_k - \vx_{k-1}\|_{\mLambda }^2.  \label{eq:descent2}
\end{align} 
where the last equality is by $\vx_{k-1,1} = \vx_{k-1}$ and the definition of $\mLambda.$

Telescoping \eqref{eq:descent2} from $i=1$ to $k$, we have
\begin{align}
F(\vx_k) \le F(\vx_0) - \sum_{i=1}^k \frac{1}{2} \|\vx_i - \vx_{i-1}\|_{\mLambda }^2.    
\end{align}

It remains to use that $F(\vx_k)\ge F(\vx^*)$, by the definition of $\vx^*$. 
\end{proof}

\thmccdpl*
\begin{proof}
Combining Lemma~\ref{lem:ccd-grad} and Inequality~\eqref{eq:descent2}, we have
\begin{equation*}
\begin{aligned}
    {\dist}^2(\partial F(\vx_k), \vzero) \le \;&  2(  \hat{L}  + 1) \|\vx_k - \vx_{k-1}\|^2_{\mLambda} \\
    \leq \;& 4(  \hat{L}  + 1) (F(\vx_{k - 1}) - F(\vx_k)).
\end{aligned}
\end{equation*}
Using Assumption~\ref{assmpt:PL-condition} and the last inequality, we have 
\begin{equation*}
    2\mu(F(\vx_k) - F(\vx^*)) \leq 4(  \hat{L}  + 1) (F(\vx_{k - 1}) - F(\vx_k)), 
\end{equation*}
which leads to 
\begin{equation*}
    F(\vx_k) - F(\vx^*) \leq \frac{2(  \hat{L}  + 1)}{2(  \hat{L}  + 1) + \mu}(F(\vx_{k - 1}) - F(\vx^*)).
\end{equation*}
Applying the last inequality recursively from $K$ down to 1, we obtain 
\begin{equation*}
    F(\vx_K) - F(\vx^*) \leq \Big(\frac{2(  \hat{L}  + 1)}{2(  \hat{L}  + 1) + \mu}\Big)^{K}(F(\vx_{0}) - F(\vx^*)), 
\end{equation*}
which completes the proof.
\end{proof}


\section{Omitted Proofs from Section~\ref{sec:vr-bccd}}\label{appx:vr-bccd}

To prove Theorem~\ref{thm:vr-main-result}, we first prove the following auxiliary lemma. 

\lemmavr*
\begin{proof}
Let $\gF_{k, j - 1}$ denote the natural filtration, containing all algorithm randomness up to and including outer iteration $k$ and inner iteration $j - 1.$ By the definition of $\vg_{k-1}^j,$ we have
\begin{align}
&\E[\|\vg_{k-1}^j - \nabla^j  f(\vx_{k-1,  j})\|_{\mLambda_j^{-1}}^2| \gF_{k, j - 1}]    \nonumber\\   
=\;&  p\E\Big[\big\|\frac{1}{b}\sum_{i=1}^b \nabla^j f_i(\vx_{k-1,  j}) - \nabla^j  f(\vx_{k-1,  j})\big\|_{\mLambda_j^{-1}}^2| \gF_{k, j - 1}\Big] \nonumber\\   
&\;+ (1-p)\E\Big[\big\|\vg_{k-2}^j + \frac{1}{b'} \sum_{i=1}^{b'}  (\nabla^j f_i(\vx_{k-1,  j}) - \nabla^j f_i(\vx_{k-2,j})) - \nabla^j  f(\vx_{k - 1,  j})\big\|_{\mLambda_j^{-1}}^2|\gF_{k, j - 1}\Big]      \nonumber\\   
=\;&p\E\Big[\big\|\frac{1}{b}\sum_{i=1}^b \nabla^j f_i(\vx_{k-1,  j}) - \nabla^j  f(\vx_{k-1,  j})\big\|_{\mLambda_j^{-1}}^2|\gF_{k, j - 1}\Big] +  (1-p)\|\vg_{k-2}^j -  \nabla^j  f(\vx_{k-2,j})\|_{\mLambda_j^{-1}}^2  \nonumber\\   
& + (1-p)\E\Big[\big\| \frac{1}{b'} \sum_{i=1}^{b'}  (\nabla^j f_i(\vx_{k-1,  j}) - \nabla^j f_i(\vx_{k-2,j})) - \nabla^j  f(\vx_{k-1,  j}) +  \nabla^j f(\vx_{k-2,j})\|_{\mLambda_j^{-1}}^2|\gF_{k, j - 1}\Big],\notag
\end{align}
where the last equality uses that $\vg_{k-2}^j -  \nabla^j  f(\vx_{k-2,j})$ is measurable w.r.t.~$\gF_{k, j - 1}$ and $\E[\frac{1}{b'} \sum_{i=1}^{b'}  (\nabla^j f_i(\vx_{k-1,  j}) - \nabla^j f_i(\vx_{k-2,j})) - \nabla^j  f(\vx_{k-1,  j}) +  \nabla^j f(\vx_{k-2,j})|\gF_{k, j - 1}] = 0.$ Using that for any random variable $X,$ $\E[(X - \E[X])^2] \leq \E[X^2]$ and proceeding as in Lemma~\ref{lem:batch-vr} with respect to $\|\cdot\|_{\mLambda_j^{-1}}$, we further have
\begin{align}
&\E[\|\vg_{k-1}^j - \nabla^j  f(\vx_{k-1,  j})\|_{\mLambda_j^{-1}}^2| \gF_{k, j - 1}]    \nonumber\\ 
\le \;& p\E\Big[\big\|\frac{1}{b}\sum_{i=1}^b \nabla^j f_i(\vx_{k-1,  j}) - \nabla^j  f(\vx_{k-1,  j})\big\|_{\mLambda_j^{-1}}^2|\gF_{k, j - 1}\Big] +  (1-p)\|\vg_{k-2}^j -  \nabla^j  f(\vx_{k-2,j})\|_{\mLambda_j^{-1}}^2   \nonumber\\  
\;&+(1-p)\E\Big[\big\| \frac{1}{b'} \sum_{i=1}^{b'}  (\nabla^j f_i(\vx_{k-1,  j}) - \nabla^j f_i(\vx_{k-2,j}))\|_{\mLambda_j^{-1}}^2 |\gF_{k, j - 1}\Big]  
\nonumber
\\  
\le\;& \frac{p(n-b)}{b(n-1)}\E_i\big[ \|\nabla^j f_i(\vx_{k-1,  j}) -  \nabla^j  f(\vx_{k-1,  j})\big\|_{\mLambda_j^{-1}}^2\big]           +  (1-p)\big\|\vg_{k-2}^j -  \nabla^j  f(\vx_{k-2,j})\|_{\mLambda_j^{-1}}^2  \nonumber\\           
\;&+\frac{1-p}{b'}\E_i[\| \nabla^j f_i(\vx_{k-1,  j}) - \nabla^j f_i(\vx_{k-2,j})\|_{\mLambda_j^{-1}}^2]. \label{eq:vr-1}
\end{align}

We now proceed to simplify Eq.~\eqref{eq:vr-1}, by simplifying the first and the last term appearing in it. For the last term, we have, using our smoothness assumption, 
\begin{align}
&\E_i\Big[\big\| \nabla^j f_i(\vx_{k-1,  j}) - \nabla^j f_i(\vx_{k-2,j})\big\|_{\mLambda_j^{-1}}^2\Big] \nonumber\\
\le\;&2\E_i\Big[\big\|\nabla^j f_i(\vx_{k-1,  j}) - \nabla^j f_i(\vx_{k-1})\big\|_{\mLambda_j^{-1}}^2\Big] + 2\E_i\Big[\big\|\nabla^j f_i(\vx_{k-1}) - \nabla^j f_i(\vx_{k-2,j})\big\|_{\mLambda_j^{-1}}^2\Big] \nonumber\\  
\le\;& 2(\vx_{k-1,  j} - \vx_{k-1})^T\mQ^j (\vx_{k-1,  j} - \vx_{k-1}) + 2 (\vx_{k-1} - \vx_{k-2,j})^T \mQ^j(\vx_{k-1} - \vx_{k-2,j})    \nonumber\\
=\;&  2 (\vx_k -  \vx_{k-1})^T \mQt^j (\vx_k -  \vx_{k-1})  + 2 (\vx_{k-1} - \vx_{k-2})^T\mQh^j(\vx_{k-1} - \vx_{k-2}), \label{eq:vr-2} 
\end{align}
where the first inequality comes from adding and subtracting $\nabla^j f_i(\vx_{k-1})$ and using Young's inequality, the second inequality is by Assumption~\ref{assmpt:vr-Q-hat-tilde}, and the last equality is by the definitions of $\vx_{k-1, j}$ and matrices $\mQh^j$ and $\mQt^j.$ 
Summing Eq.~\eqref{eq:vr-2} from $j=1$ to $m$, it follows that 
\begin{align}
& \sum_{j=1}^m  \E_i\Big[ \big\| \nabla^j  f_i(\vx_{k-1,  j}) -\nabla^j  f_i(\vx_{k-2,j})\big\|_{\mLambda_j^{-1}}^2\Big] \nonumber\\   
\leq\;&  2(\vx_k -  \vx_{k-1})^T\Big(\sum_{j=1}^m  \mQt^j\Big) (\vx_k -  \vx_{k-1})  +  2(\vx_{k-1} - \vx_{k-2})^T\Big(\sum_{j=1}^m  \mQh^j\Big)(\vx_{k-1} - \vx_{k-2}) \nonumber\\ 
\le\;&2\Big\| \mLambda^{-1/2}\Big( \sum_{j=1}^m   \mQt^j\Big)\mLambda^{-1/2}\Big\|\|\vx_k -  \vx_{k-1}\|_{\mLambda}^2 +  2\Big\| \mLambda^{-1/2}\Big( \sum_{j=1}^m \mQh^j\Big) \mLambda^{-1/2}\Big\|\|\vx_{k-1} - \vx_{k-2}\|_{\mLambda}^2\nonumber\\ 
=\;&2\tilde{L}\|\vx_k -  \vx_{k-1}\|_{\mLambda}^2  + 2\hat{L}\|\vx_{k-1} - \vx_{k-2}\|_{\mLambda}^2. \nonumber
\end{align}
Taking expectation with all the randomness in the algorithm on both sides and applying the tower property of expectation, we have 
\begin{equation}\label{eq:vr-3} 
    \E\Big[\sum_{j=1}^m \big\| \nabla^j  f_i(\vx_{k-1,  j}) -\nabla^j  f_i(\vx_{k-2,j})\big\|_{\mLambda_j^{-1}}^2\Big] \leq 2\tilde{L}\E\|\vx_k -  \vx_{k-1}\|_{\mLambda}^2  + 2\hat{L}\E\|\vx_{k-1} - \vx_{k-2}\|_{\mLambda}^2.
\end{equation}
On the other hand, to simplify the first term, we add and subtract $\nabla^j f_i(\vx_{k-1}) + \nabla^j f(\vx_{k-1})$ and apply Young's inequality to get
\begin{align}
&\E_i\Big[ \big\|\nabla^j f_i(\vx_{k-1,  j}) - \nabla^j f(\vx_{k-1,j})\big\|_{\mLambda_j^{-1}}^2 \Big]\nonumber\\   
=\;& \E_i\Big[ \|\nabla^j f_i(\vx_{k-1,j}) -  \nabla^j f_i(\vx_{k-1}) +   \nabla^j f(\vx_{k-1}) -  \nabla^j f(\vx_{k-1,j}) + \nabla^j f_i(\vx_{k-1}) -  \nabla^j f(\vx_{k-1})\|_{\mLambda_j^{-1}}^2\Big]  \nonumber\\   
\le \;&2\E_i\Big[\|\nabla^j f_i(\vx_{k-1,  j}) -  \nabla^j f_i(\vx_{k-1}) -\big(  \nabla^j f(\vx_{k-1,j}) -  \nabla^j f(\vx_{k-1})\big) \|_{\mLambda_j^{-1}}^2 + \|\nabla^j f_i(\vx_{k-1}) -  \nabla^j f(\vx_{k-1})\|_{\mLambda_j^{-1}}^2\Big].  \nonumber
 \end{align}
 Similarly as before, using that the variance of any random variable is bounded by its second moment, and summing from $j = 1$ to $m$, we further have
 \begin{align}
 &\sum_{j=1}^m\E_i\Big[ \big\|\nabla^j f_i(\vx_{k-1,  j}) - \nabla^j f(\vx_{k-1,j})\big\|_{\mLambda_j^{-1}}^2 \Big]\nonumber\\  
\le \;&2\sum_{j=1}^m\E_i\Big[ \|\nabla^j f_i(\vx_{k-1,  j}) -  \nabla^j f_i(\vx_{k-1})\|_{\mLambda_j^{-1}}^2+  \|\nabla^j f_i(\vx_{k-1}) -  \nabla^j f(\vx_{k-1})\|_{\mLambda_j^{-1}}^2\Big]   \nonumber\\
= \;& 2\sum_{j=1}^m\E_i\Big[\|\nabla^j f_i(\vx_{k-1,  j}) -  \nabla^j f_i(\vx_{k-1})\|_{\mLambda_j^{-1}}^2\Big] + 2\E_i\Big[\|\nabla f_i(\vx_{k-1}) -  \nabla f(\vx_{k-1})\|_{\mLambda^{-1}}^2\Big]\nonumber\\
\le \;&2\tilde{L} \|\vx_k - \vx_{k-1}\|_{\mLambda}^2 + 2\sigma^2, \nonumber
\end{align}
where in the last inequality we used the arguments as in deriving \eqref{eq:vr-2} and \eqref{eq:vr-3}, and Assumption~\ref{assmpt:finite-var}. Further, taking expectation with all the randomness in the algorithm on both sides and using the tower property of expectation, we have 
\begin{equation}\label{eq:vr-4}
    \E\Big[\sum_{j=1}^m \big\|\nabla^j f_i(\vx_{k-1,  j}) - \nabla^j f(\vx_{k-1,j})\big\|_{\mLambda_j^{-1}}^2\Big] \leq 2\tilde{L} \E\|\vx_k - \vx_{k-1}\|_{\mLambda}^2 + 2\sigma^2.
\end{equation}

Summing Eq.~\eqref{eq:vr-1} from $j=1$ to $m$, taking the expectation w.r.t.~all the randomness in the algorithm on both sides, and applying linearity and the tower property of expectation $\E[\E[\cdot|\gF_{k, j - 1}]] = \E[\cdot]$ for each $j \in [m]$, we have 
\begin{equation}\label{eq:vr-5} 
    \begin{aligned}
        \E\Big[\sum_{j = 1}^m \|\vg_{k-1}^j - \nabla^j  f(\vx_{k-1,  j})\|_{\mLambda_j^{-1}}^2\Big] \leq \;& \frac{p(n-b)}{b(n-1)}\E\Big[\sum_{j = 1}^m \|\nabla^j f_i(\vx_{k-1,  j}) -  \nabla^j  f(\vx_{k-1,  j})\Big\|_{\mLambda_j^{-1}}^2\Big] \\
        & + (1-p)\E\Big[\sum_{j = 1}^m\|\vg_{k-2}^j -  \nabla^j  f(\vx_{k-2,j})\|_{\mLambda_j^{-1}}^2\Big] \\
        & + \frac{1-p}{b'}\E\Big[\sum_{j = 1}^m\| \nabla^j f_i(\vx_{k-1,  j}) - \nabla^j f_i(\vx_{k-2,j})\|_{\mLambda_j^{-1}}^2\Big]
    \end{aligned}
\end{equation}
To complete the proof, it remains to plug Inequalities~\eqref{eq:vr-3}--\eqref{eq:vr-4} into Inequality~\eqref{eq:vr-5} and do some rearrangements.
\end{proof}

\lemmavrdescent*
\begin{proof}
By Assumption \ref{assmpt:coordinate-smooth} and the definitions of $\vx_{k-1,  j+1}$ and $\vx_{k-1,  j}$, we have 
\begin{align}
\;& F(\vx_{k-1,  j+1}) = f(\vx_{k-1,  j+1}) + r(\vx_{k-1,  j+1}) \nonumber\\   
\le\;& f(\vx_{k-1,  j}) + \langle \nabla^j  f(\vx_{k-1,  j}), \vx_k^j - \vx_{k-1}^j\rangle + \frac{1}{2}\|\vx_k^j - \vx_{k-1}^j\|_{\mLambda_j}^2 + \sum_{i=1}^{j-1} r^i(\vx_k^i)  + r^j(\vx_k^j) + \sum_{i=j+1}^{m} r^i(\vx_{k-1}^i) \nonumber\\ 
\le \;& f(\vx_{k-1,  j}) + r(\vx_{k-1,  j})  +  \langle \nabla^j  f(\vx_{k-1,  j})  + r^{j, '}(\vx_k^j), \vx_k^j - \vx_{k-1}^j\rangle + \frac{1}{2}\|\vx_k^j - \vx_{k-1}^j\|_{\mLambda_j}^2  \nonumber\\  
 = \;& F(\vx_{k-1,  j}) +  \langle\vg_{k-1}^j +  r^{j, '}(\vx_k^j), \vx_k^j - \vx_{k-1}^j\rangle + \frac{1}{2}\|\vx_k^j - \vx_{k-1}^j\|_{\mLambda_j}^2 +  \langle \nabla^j  f(\vx_{k-1,  j}) - \vg_{k-1}^j, \vx_k^j - \vx_{k-1}^j\rangle.  \nonumber  
  \end{align}
  Since, by assumption, $r^{j, '}(\vx_{k}^j) \in \partial r^j(\vx_{k}^j)$ is such that $\vx_k^j = \vx_{k-1}^j - \eta\mLambda_j^{-1}(\vg_{k-1}^j + r^{j, '}(\vx_k^j)),$ we further have
  \begin{align}
   F(\vx_{k-1,  j+1}) 
\leq\;&  F(\vx_{k-1,  j}) -\big(\frac{1}{\eta} -\frac{1}{2}\big)\|\vx_k^j - \vx_{k-1}^j\|_{\mLambda_j}^2 +  \Big\langle \nabla^j  f(\vx_{k-1,  j}) - \vg_{k-1}^j, -  \eta\mLambda_j^{-1}  (\vg_{k-1}^j + r^{j, '}(\vx_k^j)) \Big\rangle. \label{eq:vr-ccd-1}
\end{align}
To simplify \eqref{eq:vr-ccd-1}, we now further simplify the last term appearing in it, as follows: 
\begin{align}
& \Big\langle \nabla^j  f(\vx_{k-1,  j}) - \vg_{k-1}^j, - \eta\mLambda_j^{-1}  (\vg_{k-1}^j +   r^{j, '}(\vx_k^j)) \Big\rangle \nonumber  \\  
=\;& \eta \| \nabla^j  f(\vx_{k-1,  j}) - \vg_{k-1}^j\|_{\mLambda_j^{-1}}^2 - \eta \langle \nabla^j  f(\vx_{k-1,  j}) - \vg_{k-1}^j, \mLambda_j^{-1}  \big(\nabla^j f(\vx_{k-1,  j}) + r^{j, '}(\vx_k^j)\big)\rangle     \nonumber  \\
= \;& \eta \| \nabla^j  f(\vx_{k-1,  j}) - \vg_{k-1}^j\|_{\mLambda_j^{-1}}^2 \nonumber  \\
& - \frac{\eta}{2}\Big(\|\nabla^j f(\vx_{k-1,  j}) - \vg_{k-1}^j\|_{\mLambda_j^{-1}}^2 +  \|\nabla^j f(\vx_{k-1,  j}) + r^{j, '}(\vx_k^j)\|_{\mLambda_j^{-1}}^2 - \|\vg_{k-1}^j+ r^{j, '}(\vx_k^j)\|_{\mLambda_j^{-1}}^2\Big)      \nonumber  \\ 
=\;& \frac{\eta}{2}\| \nabla^j  f(\vx_{k-1,  j}) - \vg_{k-1}^j\|_{\mLambda_j^{-1}}^2 - \frac{\eta}{2}\|\nabla^j f(\vx_{k-1,  j}) + r^{j, '}(\vx_k^j)\|_{\mLambda_j^{-1}}^2 + \frac{\eta}{2}\|\vg_{k-1}^j + r^{j, '}(\vx_k^j) \|_{\mLambda_j^{-1}}^2\nonumber  \\
=\;& \frac{\eta}{2}\| \nabla^j  f(\vx_{k-1,  j}) - \vg_{k-1}^j\|_{\mLambda_j^{-1}}^2 -\frac{\eta}{2} \|\nabla^j f(\vx_{k-1,  j}) + r^{j, '}(\vx_k^j)\|_{\mLambda_j^{-1}}^2 + \frac{1}{2\eta} \|\vx_k^j - \vx_{k-1}^j\|_{\mLambda_j}^2.  \label{eq:vr-ccd-2}
\end{align}
Thus, combining \eqref{eq:vr-ccd-1} and \eqref{eq:vr-ccd-2}, we have 
\begin{align}
 F(\vx_{k-1,  j+1}) \le\;&  F(\vx_{k-1,  j}) -\frac{1}{2}\big(\frac{1}{\eta} - 1\big) \|\vx_k^j - \vx_{k-1}^j\|_{\mLambda_j}^2 + \frac{\eta}{2} \| \nabla^j  f(\vx_{k-1,  j}) - \vg_{k-1}^j\|_{\mLambda_j^{-1}}^2 \nonumber  \\
\;&  - \frac{\eta}{2} \|\nabla^j f(\vx_{k-1,  j}) + r^{j, '}(\vx_k^j)\|_{\mLambda_j^{-1}}^2. \label{eq:vr-ccd-22}
\end{align}

Summing \eqref{eq:vr-ccd-22} from $j=1$ to $m$, 
\begin{align}
 F(\vx_{k-1,  m+1}) \le\;&  F(\vx_{k-1, 1})  - \frac{1}{2} \big(\frac{1}{\eta} - 1\big)\|\vx_k - \vx_{k-1}\|_{\mLambda}^2 + \frac{\eta}{2}\sum_{j=1}^m \| \nabla^j  f(\vx_{k-1,  j}) - \vg_{k-1}^j\|_{\mLambda_j^{-1}}^2 \nonumber  \\
& - \frac{\eta}{2}\sum_{j=1}^m\|\nabla^j f(\vx_{k-1,  j}) + r^{j, '}(\vx_k^j)\|_{\mLambda_j^{-1}}^2. \label{eq:vr-ccd-5}
\end{align}
To complete the proof, it remains to observe that $\vx_k = \vx_{k-1,  m+1}$ and $\vx_{k-1} = \vx_{k-1, 1}.$ 
\end{proof}

\lemmavrgradnorm*
\begin{proof}
Observe first that  
\begin{align}
s_k \leq &\;\|\nabla f(\vx_k) + r'(\vx_k)\|_{\mLambda^{-1}}^2 \nonumber\\
= \;&\sum_{j=1}^m  \| \nabla^j f(\vx_k) + r^{j,'}(\vx_k)\|_{\mLambda_j^{-1}}^2    \nonumber\\
\le\;& \sum_{j=1}^m 2 \big(\|\nabla^j f(\vx_k) - \nabla^j f(\vx_{k-1,j})\|_{\mLambda_j^{-1}}^2 + \| \nabla^j f(\vx_{k-1,j}) + r^{j,'}(\vx_k)\|_{\mLambda_j^{-1}}^2\big) ,\label{eq:last-lem-1}
\end{align}
where we have used Young's inequality. On the other hand, by Assumption~\ref{assmpt:vr-Q-hat-tilde} and the definitions of $\vx_{k-1,j}$ and $\mQh^j,$ we have 
\begin{align}
\sum_{j=1}^m \|\nabla^j f(\vx_k) - \nabla^j f(\vx_{k-1,j})\|_{\mLambda_j^{-1}}^2 
\le\;& \sum_{j=1}^m (\vx_k - \vx_{k-1, j})^T \mQ^j (\vx_k - \vx_{k-1, j})\notag\\
=&\; \sum_{j=1}^m (\vx_k - \vx_{k-1})^T \mQh^j (\vx_k - \vx_{k-1}) \notag\\
\leq&\; \Big\|\mLambda^{-1/2}  \sum_{j=1}^m \hat{\mQ}^j  \mLambda^{-1/2}\Big\|\|\vx_k - \vx_{k-1}\|_{\mLambda}^2
 \nonumber\\
=\;&  \hat{L} \|\vx_k - \vx_{k-1}\|_{\mLambda}^2.  \label{eq:last-lem-2}
\end{align}
Combining \eqref{eq:last-lem-1} and \eqref{eq:last-lem-2} completes the proof. 
\end{proof}

\theoremvrmainresult*
\begin{proof}
Combining Lemmas \ref{lem:vr-descent} and \ref{lem:vr-grad-norm}, we have 
\begin{align}
F(\vx_{k}) \le\;&  F(\vx_{k-1})  - \frac{1}{2} \big(\frac{1}{\eta} - 1 - \hat{L}\eta\big)v_k + \frac{\eta}{2}u_k - \frac{\eta}{4}s_k. \label{eq:vr-main-result1}
\end{align}
For notational convenience, denote $r_k = F(\vx_k)$; then 
Eq.~\eqref{eq:vr-main-result1} is equivalent to 
\begin{align}
r_k \le r_{k-1} -  \frac{1}{2} \big(\frac{1}{\eta} - 1 - \hat{L}\eta \big)v_k + \frac{\eta}{2}u_k - \frac{\eta}{4}s_k.  \label{eq:vr-main-result3}
\end{align}
Taking expectation on both sides of Eq.~\eqref{eq:vr-main-result3} with respect to all randomness of the algorithm, and adding $\frac{\eta}{2p}\times \eqref{eq:vr-bound}$ to \eqref{eq:vr-main-result3}, we have 
\begin{align}
\E[r_k] +  \frac{(1-p)\eta}{2p} \E[u_k] \le\;& \E[r_{k-1}] + \frac{(1-p)\eta}{2p}\E[u_{k-1}] + \frac{(n-b)\sigma^2 \eta}{b(n-1)}  - \frac{\eta}{4}\E[s_k]  \nonumber \\
& - \frac{1}{2}\bigg(\frac{1}{\eta} - 1 - \hat{L}\eta - 4\Big(\frac{p(n-b)}{b(n-1)}+\frac{1-p}{b'}\Big)\frac{\tilde{L} \eta}{2p} \bigg)\E[v_k] + \frac{(1-p)\hat{L}\eta}{pb'}\E[v_{k-1}].      \nonumber
\end{align}
Observe that in the last inequality, the terms corresponding to $\E[r_k]$ and $\E[u_k]$ telescope. To further simplify this inequality, we now make a choice of $\eta$ that ensures that the terms that correspond to $\E[v_k]$ telescope as well. In particular, for the terms corresponding to $\E[v_k]$ and $\E[v_{k-1}]$ to be telescoping, we need
\begin{align}\label{eq:vk-telescoping-cond}
 \frac{1}{\eta} - 1 - \hat{L}\eta - 4  \Big(\frac{p(n-b)}{b(n-1)}+\frac{1-p}{b'}\Big)\frac{\tilde{L} \eta}{2p} \ge \frac{2(1-p)\hat{L}\eta}{p b'}.
\end{align}
To simplify \eqref{eq:vk-telescoping-cond}, 
denote by $c_0 = \frac{2(1-p)\hat{L}}{pb'} + \hat{L} + 2\Big( \frac{p(n-b)}{b(n-1)}+\frac{1-p}{b'}  \Big)\frac{\tilde{L}}{p}$ the coefficient multiplying $\eta$. Then, solving \eqref{eq:vk-telescoping-cond}, which is a quadratic inequality in $\eta,$ we get that it suffices to have  
\begin{align*}
0<\eta \le \frac{-1 + \sqrt{1+4c_0}}{2c_0},      
\end{align*}
as required by the theorem assumptions. 
Thus, we obtain 
\begin{align}
&\frac{\eta}{4}\E[s_k] +  \E[r_k] +  \frac{(1-p)\eta}{2p}\E[u_k] + \frac{(1-p)\hat{L}\eta}{pb'}\E[v_k]\nonumber  \\  
\le\;& \E[r_{k-1}] +  \frac{(1-p)\eta}{2p}\E[u_{k-1}]  
  +  \frac{(1-p)\hat{L}\eta}{pb'}\E[v_{k-1}]
  + \frac{(n-b)\sigma^2\eta}{b(n-1)}.  \label{eq:vr-main-result5}
\end{align}
Telescoping Eq.~\eqref{eq:vr-main-result5} from $1$ to $k,$ we have 
\begin{align}
 &\sum_{i=1}^k \frac{\eta}{4}\E[s_i] +  \E[r_k] + \frac{(1-p)\eta}{2p}\E[u_k] + \frac{(1-p)\hat{L}\eta}{pb'}\E[v_k]\nonumber  \\  
\le\;& \E[r_{0}] + \frac{(1-p)\eta}{2p}\E[u_{0}]  
  +  \frac{(1-p)\hat{L}\eta}{pb'}\E[v_{0}]
  + \frac{(n-b)\sigma^2\eta k}{b(n-1)}.  \label{eq:vr-main-result6}
\end{align}

With the fact that $r_k = F(\vx_k) \ge F(\vx^*), u_k \ge 0, v_k \ge 0$, $\E[u_0 ]= \E\Big[ \sum_{j=1}^m \| \nabla^j  f(\vx_{-1,  j}) - \vg_{-1}^j\|_{\mLambda_j^{-1}}^2\Big]   \le \frac{(n-b)\sigma^2}{b(n-1)}$ (by Lemma~\ref{lem:batch-vr} adapted to $\mLambda_j$ norms),   $v_0= \|\vx_0 - \vx_{-1}\|_{\mLambda}^2 = 0$ and the definition of $s_k,$ we have 
\begin{align}\label{ineq:sum-rate}
\sum_{i=1}^k \frac{\eta}{4}\E\Big[{\dist}^2(\partial F(\vx_i), \vzero)\Big] \le F(\vx_0) - F(\vx^*) +   \frac{(1-p)(n-b)\eta\sigma^2}{2pb(n-1)} + \frac{(n-b)\sigma^2\eta k}{b(n-1)}.     
\end{align}
Further, taking $\hat{\vx}_K$ to be uniformly at random chosen from $\{\vx_k\}_{k \in [K]}$, we have 
\begin{equation*}
\E_k\Big[{\dist}^2(\partial F(\hat{\vx}_K), \vzero)\Big] = \frac{1}{K}\sum_{i = 1}^K{\dist}^2(\partial F(\vx_i), \vzero).
\end{equation*}
Taking expectation w.r.t.~all the randomness up to iteration $K$ on both sides and using Inequality~\eqref{ineq:sum-rate}, then we have 
\begin{equation*}
    \begin{aligned}
        \E\Big[{\dist}^2(\partial F(\hat{\vx}_K), \vzero)\Big] = \;& \frac{1}{K}\sum_{i = 1}^K\E\Big[{\dist}^2(\partial F(\vx_i), \vzero)\Big] \\
        \leq \;& \frac{4\Delta_0}{\eta K} + \frac{2(1-p)(n-b)\sigma^2}{pb(n-1) K} + \frac{4(n-b)\sigma^2}{b(n-1)}, 
    \end{aligned}
\end{equation*}
where $\Delta_0 = F(\vx_0) - F(\vx^*)$, thus completing the proof. 
\end{proof}

\finiteComplexity*
\begin{proof}
By the chosen parameters and Inequality~\eqref{ineq:rate}, we have 
\begin{equation*}
\begin{aligned}
\E\Big[{\dist}^2(\partial F(\hat{\vx}_K), \vzero)\Big] \leq \frac{4(F(\vx_0) - F(\vx^*))}{\eta K} + \frac{2(1-p)(n-b)\sigma^2}{pb(n-1) K} + \frac{4(n-b)\sigma^2}{b(n-1)} \leq \epsilon^2.
\end{aligned}
\end{equation*}
Further, since $0 < \eta \leq \frac{-1 + \sqrt{1 + 4c_0}}{2c_0}$, we have $K = \frac{4(F(\vx_0) - F(\vx^*))}{\epsilon^2\eta} = \frac{8c_0(F(\vx_0) - F(\vx^*))}{\epsilon^2(-1 + \sqrt{1 + 4c_0})}$. Let $\gF_{k, j - 1}$ denote the natural filtration, containing all algorithm randomness up to and including outer iteration $k$ and inner iteration $j - 1$. Denote $m_k^j$ to be the number of arithmetic operations to update the $j$-th block at $k$-th iteration, then we have for $k \geq 1$
\begin{equation*}
    \E[m_k^j | \gF_{k, j - 1}] = \mathcal{O}\Big((pb + (1 - p)b')d^j\Big).
\end{equation*}
Taking expectation w.r.t to all randomness on both sides, we obtain $\E[m_k^j] = \mathcal{O}\Big((pb + (1 - p)b')d^j\Big).$ Let $m_k$ be the number of arithmetic operations in the $k$-the iteration, then we have for $k \geq 1$
\begin{equation*}
    \E[m_k] = \E\Big[\sum_{j = 1}^m m_k^j\Big] = \mathcal{O}\Big((pb + (1 - p)b')\sum_{j = 1}^md^j\Big) = \mathcal{O}\Big((pb + (1 - p)b')d\Big).
\end{equation*}
Hence, the total number of arithmetic operations $M$ in $K$ iterations to obtain $\epsilon$-accurate solution is 
\begin{equation*}
    \E[M] = \E\Big[\sum_{k = 0}^K m_k\Big] = \mathcal{O}(bd) + \E\Big[\sum_{k = 1}^K m_k\Big] = \mathcal{O}\Big(bd + K(pb + (1 - p)b')d\Big).
\end{equation*}
Since $b = n$, $b' = \sqrt{b}$ and $p = \frac{b'}{b + b'}$, then we have 
\begin{equation*}
K(pb + (1 - p)b') = \frac{8c_0(F(\vx_0) - F(\vx^*))}{\epsilon^2(-1 + \sqrt{1 + 4c_0})}\frac{2bb'}{b + b'} \leq \frac{2(\sqrt{1 + 4c_0} + 1)(F(\vx_0) - F(\vx^*))}{\epsilon^2}2b'. 
\end{equation*}
Note that 
\begin{equation*}
c_0 = \frac{2(1-p)\hat{L}}{pb'} + \hat{L} + 2\Big( \frac{p(n-b)}{b(n-1)}+\frac{1-p}{b'}\Big)\frac{\tilde{L}}{p} = \hat{L} + \frac{1 - p}{pb'}(2\hat{L} + 2\tilde{L}) = \hat{L} + \frac{b}{b'^2}(2\hat{L} + 2\tilde{L}) = 3\hat{L} + 2\tilde{L} = \mathcal{O}(\hat{L} + \tilde{L}),
\end{equation*}
so we obtain 
\begin{equation*}
    K(pb + (1 - p)b') \leq \frac{2(\sqrt{1 + 4c_0} + 1)(F(\vx_0) - F(\vx^*))}{\epsilon^2}2b' = \mathcal{O}\Big(\frac{(F(\vx_0 - \vx^*))\sqrt{n(\hat{L} + \tilde{L})}}{\epsilon^2}\Big), 
\end{equation*}
thus completing the proof.
\end{proof}

\onlineComplexity*
\begin{proof}
By the chosen parameters and Inequality~\eqref{ineq:rate}, we have 
\begin{equation*}
\begin{aligned}
\E\Big[{\dist}^2(\partial F(\hat{\vx}_K), \vzero)\Big] \leq \frac{4(F(\vx_0) - F(\vx^*))}{\eta K} + \frac{2(1-p)(n-b)\sigma^2}{pb(n-1) K} + \frac{4(n-b)\sigma^2}{b(n-1)} \leq \frac{\epsilon^2}{3} + \frac{4\sigma^2}{b} + \frac{4\sigma^2}{b} \leq \epsilon^2.
\end{aligned}
\end{equation*}
Further, same as in the proof of  Corollary~\ref{cor:finite-sum}, we have that the total number of arithmetic operations $M$ in $K$ iterations to obtain an $\epsilon$-accurate solution is 
\begin{equation*}
    \E[M] = \E\Big[\sum_{k = 0}^K m_k\Big] = \mathcal{O}(bd) + \E\Big[\sum_{k = 1}^K m_k\Big] = \mathcal{O}\Big(bd + K(pb + (1 - p)b')d\Big).
\end{equation*}
Since $0 < \eta \leq \frac{-1 + \sqrt{1 + 4c_0}}{2c_0}$, we have 
\begin{equation*}
K = \frac{12(F(\vx_0) - F(\vx^*))}{\epsilon^2\eta} + \frac{1}{2p} = \frac{24c_0(F(\vx_0) - F(\vx^*))}{\epsilon^2(-1 + \sqrt{1 + 4c_0})} + \frac{b + b'}{2b'} = \frac{6(\sqrt{4c_0 + 1} + 1)(F(\vx_0) - F(\vx^*))}{\epsilon^2} + \frac{b + b'}{2b'},
\end{equation*}
which leads to 
\begin{equation*}
    K(pb + (1 - p)b') = \Big(\frac{6(\sqrt{4c_0 + 1} + 1)(F(\vx_0) - F(\vx^*))}{\epsilon^2} + \frac{b + b'}{2b'}\Big)\frac{2bb'}{b + b'} \leq b + \frac{12b'(\sqrt{1 + 4c_0} + 1)(F(\vx_0) - F(\vx^*))}{\epsilon^2}. 
\end{equation*}
Since $\frac{p(n - b)}{b(n - 1)} \leq \frac{p}{b} = \frac{b'}{b(b + b')} \leq \frac{b}{b'(b + b')} = \frac{1 - p}{b'}$ with $b' = \sqrt{b}$ and $p = \frac{b'}{b + b'}$, we have 
\begin{equation*}
c_0 = \frac{2(1-p)\hat{L}}{pb'} + \hat{L} + 2\Big( \frac{p(n-b)}{b(n-1)}+\frac{1-p}{b'}\Big)\frac{\tilde{L}}{p} \leq \hat{L} + \frac{1 - p}{pb'}(2\hat{L} + 4\tilde{L}) \leq \hat{L} + \frac{b}{b'^2}(2\hat{L} + 4\tilde{L}) = 3\hat{L} + 4\tilde{L} = \mathcal{O}(\hat{L} + \tilde{L}).
\end{equation*}
Hence, we obtain 
\begin{equation*}
    K(pb + (1 - p)b') \leq b + \frac{12b'(\sqrt{1 + 4c_0} + 1)(F(\vx_0) - F(\vx^*))}{\epsilon^2} = \mathcal{O}\Big(b + \frac{(F(\vx_0 - \vx^*))\sqrt{b(\hat{L} + \tilde{L})}}{\epsilon^2}\Big), 
\end{equation*}
thus completing the proof.
\end{proof}

\thmccdvrpl*
\begin{proof}
By Lemma~\ref{lem:vr-grad-norm} and Lemma~\ref{lem:vr-descent}, we have 
\begin{equation*}
\begin{aligned}
    F(\vx_{k}) \le\;&  F(\vx_{k-1})  - \frac{1}{2} \big(\frac{1}{\eta} - 1 - \hat{L}\eta\big)v_k + \frac{\eta}{2}u_k - \frac{\eta}{4}s_k.
\end{aligned}
\end{equation*}
By the P{\L} condition, 
\begin{equation*}
   s_k \geq 2\mu(F(\vx) - F(\vx^*)), 
\end{equation*}
we obtain 
\begin{equation*}
\begin{aligned}
    \big(1 + \frac{\eta\mu}{2}\big)(F(\vx_k) - F(\vx^*)) \leq \;& F(\vx_{k - 1}) - F(\vx^*) - \frac{1}{2}\big(\frac{1}{\eta} - 1 - \hat{L}\eta\big)v_k + \frac{\eta}{2}u_k.
\end{aligned}
\end{equation*}
Denoting $r_k = F(\vx_k) - F(\vx^*)$ and taking expectation with respect to all randomness on both sides, and adding $\frac{\eta}{p} \times$~\eqref{eq:vr-bound} in Lemma~\ref{lem:vr}, we obtain 
\begin{equation*}
\begin{aligned}
    \big(1 + \frac{\eta\mu}{2}\big)\ee[r_k] + \frac{\eta(2 - p)}{2p}\ee[u_k] \leq \;& \ee[r_{k - 1}] - \Big[\frac{1}{2}\big(\frac{1}{\eta} - 1 - \hat{L}\eta\big) - \frac{2\eta\tilde{L}}{p}\big(\frac{p(n - b)}{b(n - 1)} + \frac{1 - p}{b'}\big)\Big]\ee[v_k] \\
    & + \frac{\eta(1 - p)}{p}\ee[u_{k - 1}] + \frac{2\eta\hat{L}(1 - p)}{pb'}\ee[v_{k - 1}] + \frac{2\eta(n - b)\sigma^2}{b(n - 1)}.
\end{aligned}
\end{equation*}
Note that when $0 < \eta \leq \frac{p}{\mu(1 - p)}$, we have $\frac{\eta(2 - p)}{2p} \geq \big(1 + \frac{\eta\mu}{2}\big)\frac{\eta(1 - p)}{p}$ and $1 + \frac{\eta\mu}{2} \leq \frac{1}{1 - p}$. On the other hand, denote $$c_0 = \hat{L} + \frac{4\hat{L}}{pb'} + \frac{4\tilde{L}}{p}\big(\frac{p(n - b)}{b(n - 1)} + \frac{1 - p}{b'}\big)$$ for simplicity, then if $\eta \leq \frac{-1 + \sqrt{1 + 4c_0}}{2c_0}$, 
we have 
\begin{equation*}
    \big(1 + \frac{\eta\mu}{2}\big)\frac{2\eta\hat{L}(1 - p)}{pb'} \leq \frac{2\eta\hat{L}}{pb'} \leq \frac{1}{2}\big(\frac{1}{\eta} - 1 - \hat{L}\eta\big) - \frac{2\eta\tilde{L}}{p}\big(\frac{p(n - b)}{b(n - 1)} + \frac{1 - p}{b'}\big). 
\end{equation*}
Hence, when choosing the stepsize $\eta$ such that 
\begin{equation*}
    0 < \eta \leq \min\Big\{\frac{p}{\mu(1 - p)}, \frac{-1 + \sqrt{1 + 4c_0}}{2c_0}\Big\}, 
\end{equation*}
we have 
\begin{equation*}
\begin{aligned}
    \;& \big(1 + \frac{\eta\mu}{2}\big)\Big[\ee[r_k] + \frac{\eta(1 - p)}{p}\ee[u_k] + \frac{2\eta\hat{L}(1 - p)}{pb'}\ee[v_k]\Big] \\
    \leq \;& \big(1 + \frac{\eta\mu}{2}\big)\ee[r_k] + \frac{\eta(2 - p)}{2p}\ee[u_k] + \Big[\frac{1}{2}\big(\frac{1}{\eta} - 1 - \hat{L}\eta\big) - \frac{2\eta\tilde{L}}{p}\big(\frac{p(n - b)}{b(n - 1)} + \frac{1 - p}{b'}\big)\Big]\ee[v_k] \\
    \leq \;& \ee[r_{k - 1}] + \frac{\eta(1 - p)}{p}\ee[u_{k - 1}] + \frac{2\eta\hat{L}(1 - p)}{pb'}\ee[v_{k - 1}] + \frac{2\eta(n - b)\sigma^2}{b(n - 1)}.
\end{aligned}
\end{equation*}
Let $\Phi_k = r_k + \frac{\eta(1 - p)}{p}u_k + \frac{2\eta\hat{L}(1 - p)}{pb'}v_k$, then we obtain 
\begin{equation*}
    \ee[\Phi_k] \leq \big(1 + \frac{\eta\mu}{2}\big)^{-1}\ee[\Phi_{k - 1}] + \big(1 + \frac{\eta\mu}{2}\big)^{-1}\frac{2\eta(n - b)\sigma^2}{b(n - 1)}.
\end{equation*}
Telescoping from $1$ to $K$, we have 
\begin{equation*}
    \ee[\Phi_K] \leq \big(1 + \frac{\eta\mu}{2}\big)^{-K}\ee[\Phi_0] + \frac{4(n - b)\sigma^2}{b\mu(n - 1)}.
\end{equation*}
As $u_K \geq 0$, $v_K \geq 0$, $\E[u_0] = \E\Big[\sum_{j = 1}^m\|\nabla^j f(\vx_{-1, j}) - g_{-1}^j\|_{\mLambda_j^{-1}}^2\Big] \leq \frac{\sigma^2(n - b)}{b(n - 1)}$ and $v_0 = \|\vx_0 - \vx_{-1}\|_\mLambda^2 = 0$, we have 
\begin{equation}\label{eq:vr-rate-pl}
    \ee[F(\vx_K) - F(\vx^*)] \leq \big(1 + \frac{\eta\mu}{2}\big)^{-K}(F(\vx_0) - F(\vx^*)) + \big(1 + \frac{\eta\mu}{2}\big)^{-K}\frac{\sigma^2\eta(1 - p)(n - b)}{pb(n - 1)} + \frac{4(n - b)\sigma^2}{b\mu(n - 1)}, 
\end{equation}
thus completing the proof.
\end{proof}
\finiteComplexityPL*
\begin{proof}
By the chosen parameters and Inequality~\eqref{eq:vr-rate-pl}, we have 
\begin{equation*}
    \begin{aligned}
    \ee[F(\vx_K) - F(\vx^*)] \leq \;& \big(1 + \frac{\eta\mu}{2}\big)^{-K}(F(\vx_0) - F(\vx^*)) + \big(1 + \frac{\eta\mu}{2}\big)^{-K}\frac{\sigma^2\eta(1 - p)(n - b)}{pb(n - 1)} + \frac{4(n - b)\sigma^2}{b\mu(n - 1)} \\
    \leq \;& \exp\Big(-\frac{n\mu}{2 + \eta\mu}K\Big)(F(\vx_0) - F(\vx^*)) \\
    \leq \;& \epsilon.
    \end{aligned}
\end{equation*}
Proceeding same as in the proof for Corollary~\ref{cor:finite-sum}, we have the total number of arithmetic operations $M$ in $K$ iterations to obtain an $\epsilon$-accurate solution is 
\begin{equation*}
    \E[M] = \E\Big[\sum_{k = 0}^K m_k\Big] = \mathcal{O}(bd) + \E\Big[\sum_{k = 1}^K m_k\Big] = \mathcal{O}\Big(bd + K(pb + (1 - p)b')d\Big).
\end{equation*}
Since $0 < \eta \leq \min\Big\{\frac{p}{\mu(1 - p)}, \frac{-1 + \sqrt{1 + 4c_0}}{2c_0}\Big\}$, we can bound $K$ by 
\begin{equation*}
    K = \Big(1 + \frac{2}{\eta\mu}\Big)\log\Big(\frac{F(\vx_0) - F(\vx^*)}{\epsilon}\Big) \leq \Big(1 + \frac{\sqrt{4c_0 + 1} + 1}{\mu} + \frac{2(1 - p)}{p}\Big)\log\Big(\frac{F(\vx_0) - F(\vx^*)}{\epsilon}\Big), 
\end{equation*}
which leads to 
\begin{equation*}
\begin{aligned}
    K(pb + (1 - p)b') \leq \;& \Big(1 + \frac{\sqrt{4c_0 + 1} + 1}{\mu} + \frac{2(1 - p)}{p}\Big)\log\Big(\frac{F(\vx_0) - F(\vx^*)}{\epsilon}\Big)\frac{2bb'}{b + b'} \\
    \leq \;& \Big(2b' + 2b'\frac{\sqrt{4c_0 + 1} + 1}{\mu} + 4b\Big)\log\Big(\frac{F(\vx_0) - F(\vx^*)}{\epsilon}\Big).
\end{aligned}
\end{equation*}
Notice that 
\begin{equation*}
    c_0 = \hat{L} + \frac{4\hat{L}}{pb'} + \frac{4\tilde{L}}{p}\big(\frac{p(n - b)}{b(n - 1)} + \frac{1 - p}{b'}\big) \leq \hat{L} + \frac{4}{pb'}(\hat{L} + \tilde{L}) = \hat{L} + \frac{4(b + b')}{b'^2}(\hat{L} + \tilde{L}) \leq \hat{L} + 8(\hat{L} + \tilde{L}) = \mathcal{O}(\hat{L} + \tilde{L}), 
\end{equation*}
so we obtain 
\begin{equation*}
    K(pb + (1 - p)b') \leq \Big(2b' + 2b'\frac{\sqrt{4c_0 + 1} + 1}{\mu} + 4b\Big)\log\Big(\frac{F(\vx_0) - F(\vx^*)}{\epsilon}\Big) = \mathcal{O}\Big(\big(\frac{\sqrt{n(\hat{L} + \tilde{L})}}{\mu} + n\big)\log\big(\frac{F(\vx_0) - F(\vx^*)}{\epsilon}\big)\Big), 
\end{equation*}
thus completing the proof. 
\end{proof}

\onlineComplexityPL*
\begin{proof}
By the chosen parameters and Inequality~\eqref{eq:vr-rate-pl}, we have 
\begin{equation*}
    \begin{aligned}
    \ee[F(\vx_K) - F(\vx^*)] \leq \;& \big(1 + \frac{\eta\mu}{2}\big)^{-K}(F(\vx_0) - F(\vx^*)) + \big(1 + \frac{\eta\mu}{2}\big)^{-K}\frac{\sigma^2\eta(1 - p)(n - b)}{pb(n - 1)} + \frac{4(n - b)\sigma^2}{b\mu(n - 1)} \\
    \leq \;& \exp\Big(-\frac{n\mu}{2 + \eta\mu}K\Big)(F(\vx_0) - F(\vx^*)) + \exp\Big(-\frac{n\mu}{2 + \eta\mu}K\Big)\frac{\sigma^2\eta(1 - p)(n - b)}{pb(n - 1)} + \frac{4(n - b)\sigma^2}{b\mu(n - 1)} \\
    \overset{(\romannumeral1)}{\leq} \;& \frac{\epsilon}{3} + \frac{\epsilon}{3} + \frac{\epsilon}{3} = \epsilon, 
    \end{aligned}
\end{equation*}
where for $(\romannumeral1)$ we use $\eta \leq \frac{p}{\mu(1 - p)}$, thus 
\begin{equation*}
    \exp\Big(-\frac{n\mu}{2 + \eta\mu}K\Big)\frac{\sigma^2\eta(1 - p)(n - b)}{pb(n - 1)} \leq \frac{\sigma^2\eta(1 - p)(n - b)}{pb(n - 1)} \leq \frac{\eta(1 - p)\mu\epsilon}{12p} \leq \frac{\epsilon}{12}.
\end{equation*}
With the same process as in the proof for Corollary~\ref{cor:finite-sum}, we have the total number of arithmetic operations $M$ in $K$ iterations to obtain an $\epsilon$-accurate solution is 
\begin{equation*}
    \E[M] = \E\Big[\sum_{k = 0}^K m_k\Big] = \mathcal{O}(bd) + \E\Big[\sum_{k = 1}^K m_k\Big] = \mathcal{O}\Big(bd + K(pb + (1 - p)b')d\Big).
\end{equation*}
Since $0 < \eta \leq \min\Big\{\frac{p}{\mu(1 - p)}, \frac{-1 + \sqrt{1 + 4c_0}}{2c_0}\Big\}$, we can bound $K$ by 
\begin{equation*}
    K = \Big(1 + \frac{2}{\eta\mu}\Big)\log\Big(\frac{3(F(\vx_0) - F(\vx^*))}{\epsilon}\Big) \leq \Big(1 + \frac{\sqrt{4c_0 + 1} + 1}{\mu} + \frac{2(1 - p)}{p}\Big)\log\Big(\frac{3(F(\vx_0) - F(\vx^*))}{\epsilon}\Big), 
\end{equation*}
which leads to 
\begin{equation*}
\begin{aligned}
    K(pb + (1 - p)b') \leq \;& \Big(1 + \frac{\sqrt{4c_0 + 1} + 1}{\mu} + \frac{2(1 - p)}{p}\Big)\log\Big(\frac{3(F(\vx_0) - F(\vx^*))}{\epsilon}\Big)\frac{2bb'}{b + b'} \\
    \leq \;& \Big(2b' + 2b'\frac{\sqrt{4c_0 + 1} + 1}{\mu} + 4b\Big)\log\Big(\frac{3(F(\vx_0) - F(\vx^*))}{\epsilon}\Big).
\end{aligned}
\end{equation*}
Notice that $\frac{p(n - b)}{b(n - 1)} \leq \frac{p}{b} = \frac{b'}{b(b + b')} \leq \frac{b}{b'(b + b')} = \frac{1 - p}{b'}$, and thus 
\begin{equation*}
    c_0 = \hat{L} + \frac{4\hat{L}}{pb'} + \frac{4\tilde{L}}{p}\big(\frac{p(n - b)}{b(n - 1)} + \frac{1 - p}{b'}\big) \leq \hat{L} + \frac{4}{pb'}(\hat{L} + 2\tilde{L}) = \hat{L} + \frac{4(b + b')}{b'^2}(\hat{L} + 2\tilde{L}) = \hat{L} + 8(\hat{L} + 2\tilde{L}) = \mathcal{O}(\hat{L} + \tilde{L}). 
\end{equation*}
Hence, we obtain 
\begin{equation*}
    K(pb + (1 - p)b') \leq \Big(2b' + 2b'\frac{\sqrt{4c_0 + 1} + 1}{\mu} + 4b\Big)\log\Big(\frac{3(F(\vx_0) - F(\vx^*))}{\epsilon}\Big) = \mathcal{O}\Big(\big(\frac{\sqrt{b(\hat{L} + \tilde{L})}}{\mu} + b\big)\log\big(\frac{\Delta_0}{\epsilon}\big)\Big), 
\end{equation*}
thus completing the proof. 
\end{proof}

\section{Additional Experiments and Discussion}\label{appx:exp}
We first present the LeNet architectures used in our experiments for MNIST and CIFAR-10 datasets as in Fig.~\ref{fig:lenet}.
\begin{figure*}[ht]
    \hspace*{\fill}
    \includegraphics[width=0.4\textwidth]{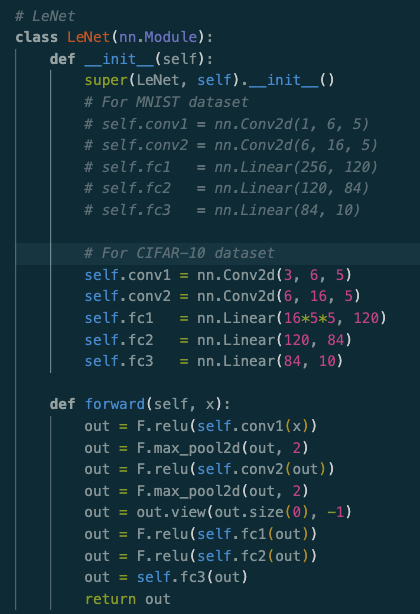}
    \hspace*{\fill}
    \caption{LeNet architectures used for MNIST and CIFAR-10 datasets.
    }
    \label{fig:lenet}
\end{figure*}

We then re-plot the train loss and test accuracy in Fig.~\ref{fig:cifar} against wall-clock time based on the average runtime per epoch of each algorithm in Table~\ref{table:time}, as is shown in Fig.~\ref{fig:cifar-time} below. We observe that (i) SCCD still converges faster to solutions with better generalization, in comparison with SGD and PAGE (whether spectral normalized or not); (ii) VRO-CCD converges slower in terms of wall-clock time, which is due to that cyclic updates become major computation bottleneck using small batch size. Some other causes can be sampling $p$ and additional batch operations to form $\gB$ and $\gB'$ in each iteration for PAGE estimator.
\begin{figure*}[ht]
    \hspace*{\fill}\subfigure[Train Loss]{\includegraphics[width=0.35\textwidth]{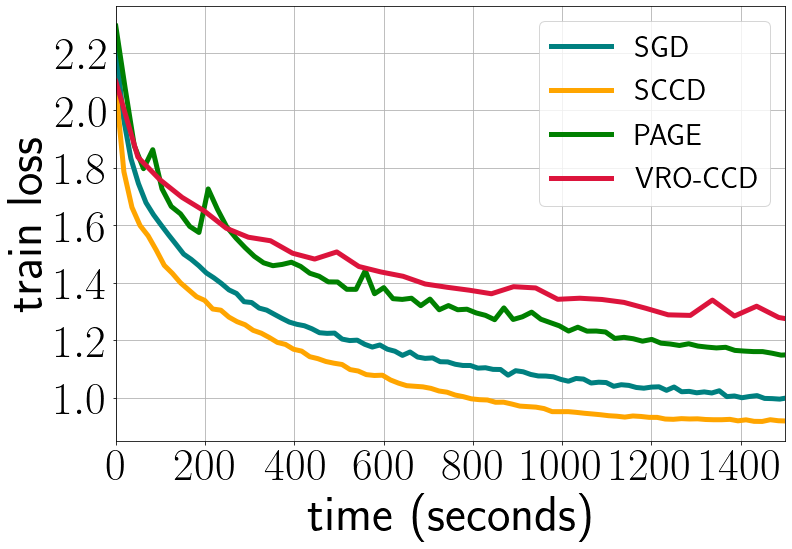}\label{fig:train-loss-time}}\hfill
    \subfigure[Test Accuracy]{\includegraphics[width=0.35\textwidth]{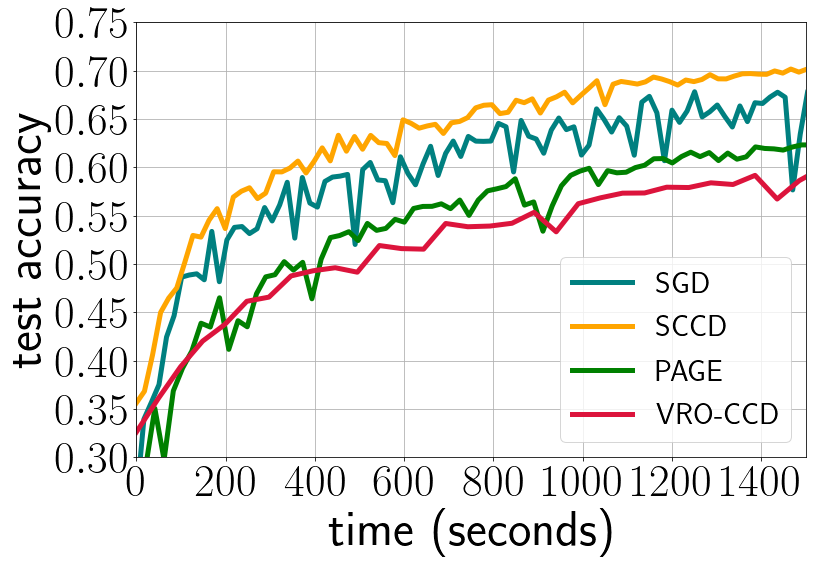}\label{fig:test-acc-time}}\hspace*{\fill}\\
    \hspace*{\fill}\subfigure[Train Loss (SN)]{\includegraphics[width=0.35\textwidth]{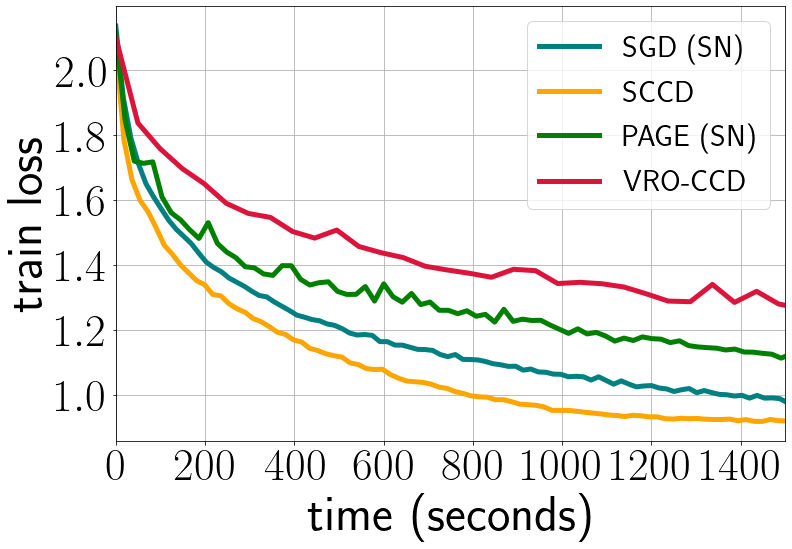}\label{fig:train-loss-sn-time}}\hfill
    \subfigure[Test Accuracy (SN)]{\includegraphics[width=0.35\textwidth]{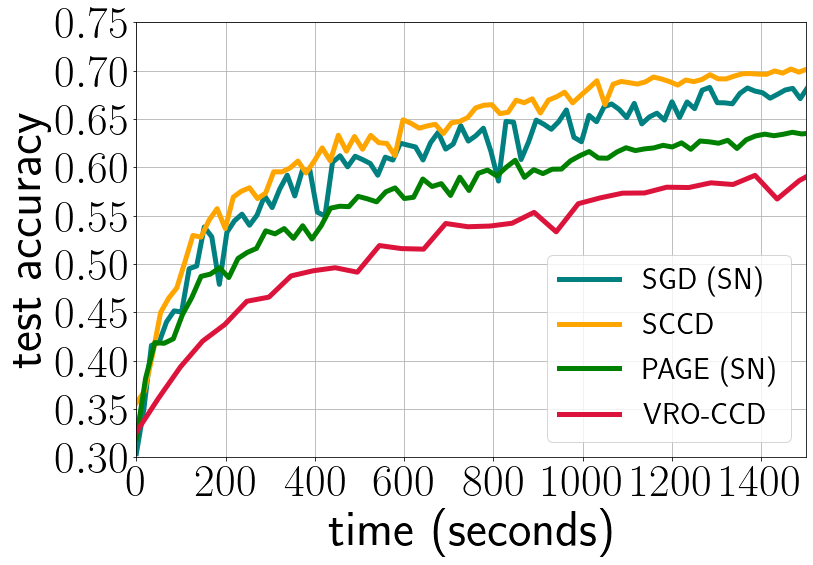}\label{fig:test-acc-sn-time}}
    \hspace*{\fill}
    \caption{Comparison of SGD, SCCD, PAGE and VRO-CCD in wall-clock time on training LeNet on CIFAR-10.}
    \label{fig:cifar-time}
\end{figure*}

We also compare the empirical performance of VR-CCD and VRO-CCD on the MNIST and CIFAR-10 datasets in Fig.~\ref{fig:vr-compare}, in one run and without spectral normalization. We set batch size $b = 64$ for the MNIST dataset and $b = 256$ for the CIFAR-10 dataset, and run for $200$ epochs. We still use the cosine learning rate scheduler, which is tuned for VR-CCD and VRO-CCD separately. We can see that (i) VR-CCD and VRO-CCD exhibits similar performance on the MNIST dataset; (ii) the empirical convergence of VR-CCD is slower than VRO-CCD's on the CIFAR-10 dataset, due to the smaller learning rate used for VR-CCD. We remark that separate sampling for each block may lead to numerical instability of the algorithm, thus requiring smaller learning rate for more complicated problems. Also, separate sampling for each block increases the sample complexity and introduces additional computational cost.
\begin{figure*}[htb]
    \hspace*{\fill}\subfigure[Train Loss (MNIST)]{\includegraphics[width=0.35\textwidth]{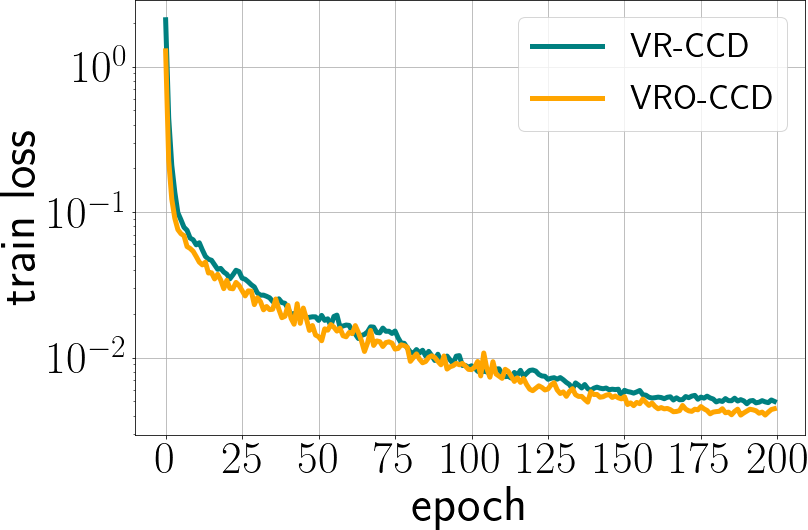}\label{fig:train-loss-mnist}}\hfill
    \subfigure[Test Accuracy (MNIST)]{\includegraphics[width=0.35\textwidth]{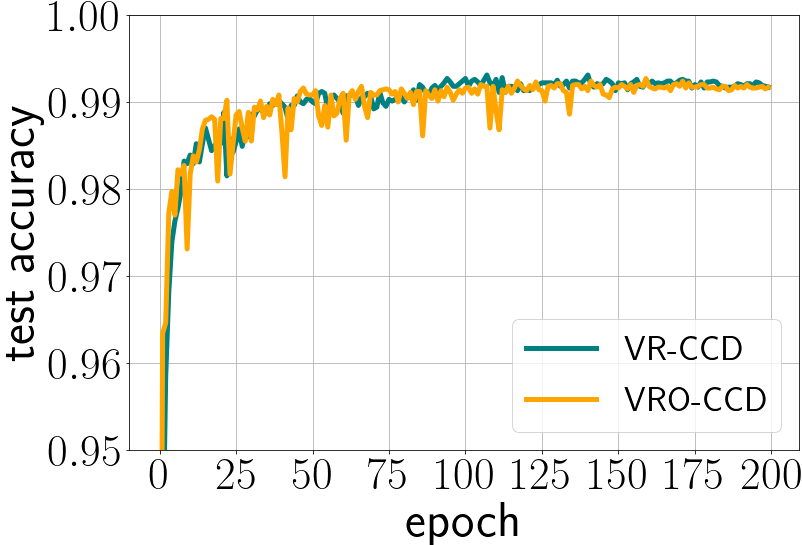}\label{fig:test-acc-mnist}}\hspace*{\fill}\\
    \hspace*{\fill}\subfigure[Train Loss (CIFAR-10)]{\includegraphics[width=0.35\textwidth]{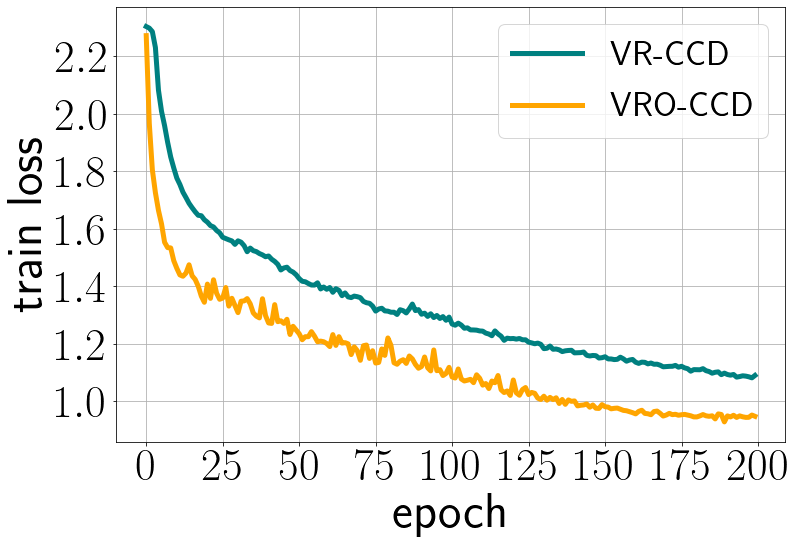}\label{fig:train-loss-cifar}}\hfill
    \subfigure[Test Accuracy (CIFAR-10)]{\includegraphics[width=0.35\textwidth]{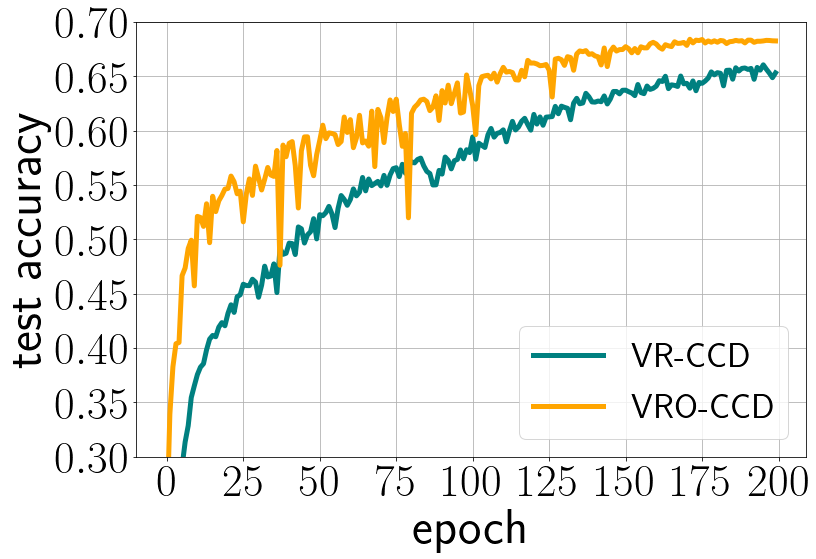}\label{fig:test-acc-cifar}}
    \hspace*{\fill}
    \caption{Comparison of VR-CCD and VRO-CCD on training LeNet on MNIST and CIFAR-10.}
    \label{fig:vr-compare}
\end{figure*}

Finally, we provide a further comparison between SCCD, VRO-CCD, SGD and PAGE algorithms with the same learning rate to illustrate the efficacy of variance reduction methods, motivated by the experimental setup in~\citet{li2020page}. We use the initial learning rate $\eta_{\text{ini}} = 0.01$ in cosine scheduler for all the algorithms except PAGE without spectral normalization which only admits $\eta_{\text{ini}} = 0.005$ most. We can see that SCCD and VRO-CCD demonstrate better performance than SGD and PAGE do, respectively. Here SCCD stagnates earlier than SGD, because the learning rate with cosine scheduler is too small for each block in SCCD to make progress.
\begin{figure*}[htb]
    \hspace*{\fill}\subfigure[Train Loss]{\includegraphics[width=0.35\textwidth]{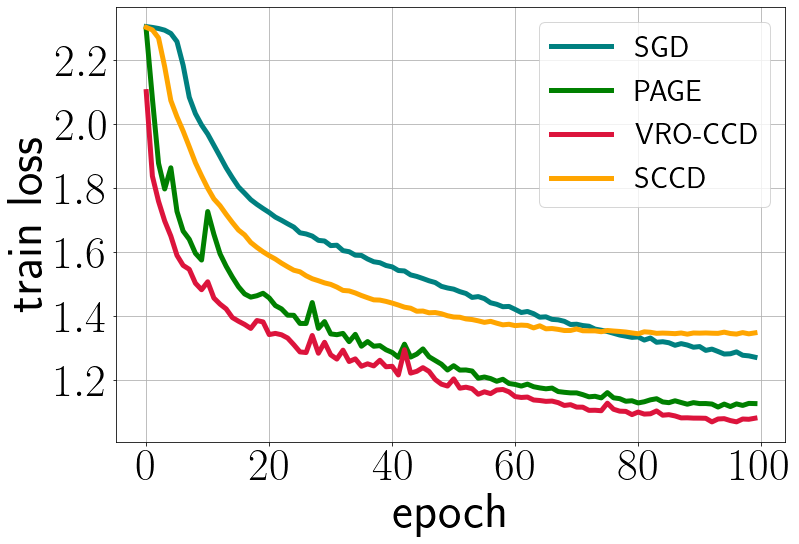}\label{fig:train-loss-same}}\hfill
    \subfigure[Test Accuracy]{\includegraphics[width=0.35\textwidth]{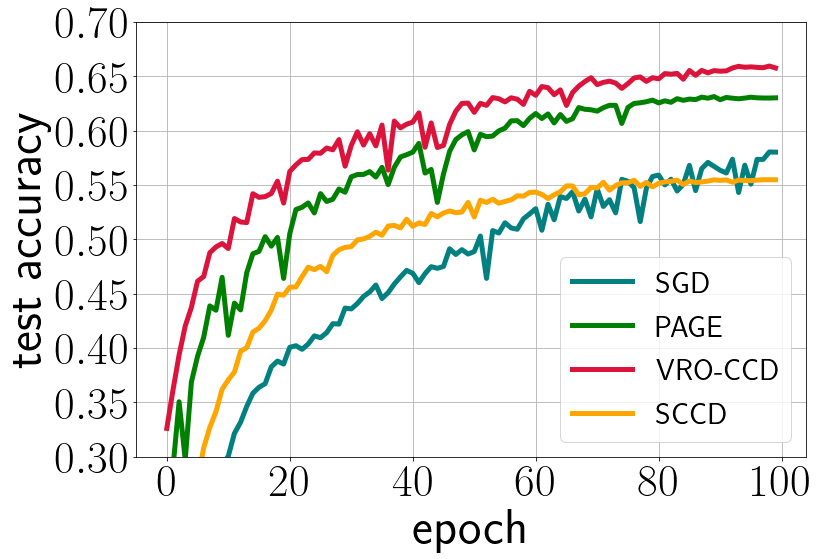}\label{fig:test-acc-same}}\hspace*{\fill}\\
    \hspace*{\fill}\subfigure[Train Loss (SN)]{\includegraphics[width=0.35\textwidth]{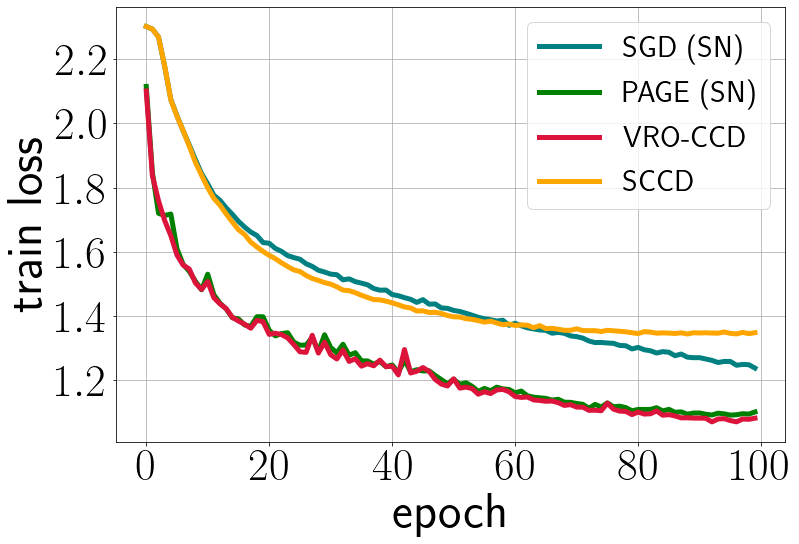}\label{fig:train-loss-sn-same}}\hfill
    \subfigure[Test Accuracy (SN)]{\includegraphics[width=0.35\textwidth]{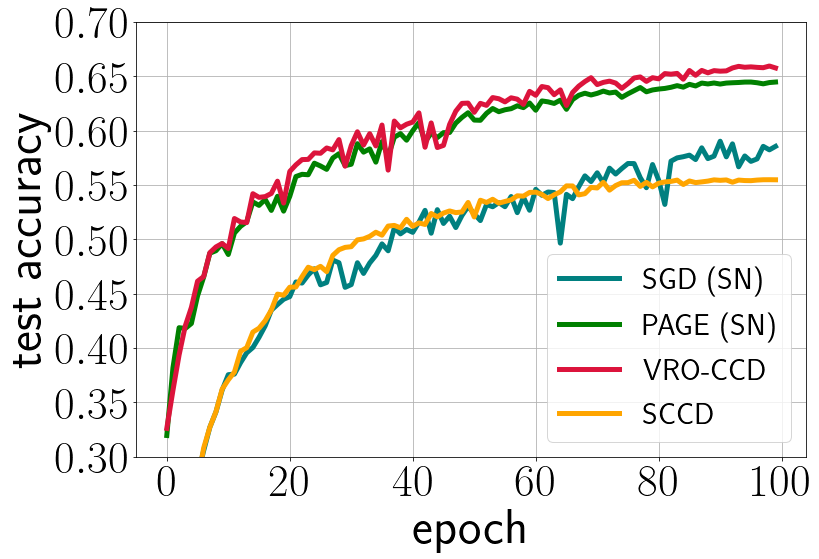}\label{fig:test-acc-sn-same}}
    \hspace*{\fill}
    \caption{Comparison of SGD, SCCD, PAGE and VRO-CCD with same learning rate on training LeNet on CIFAR-10.}
    \label{fig:cifar-same}
\end{figure*}

\end{document}